\documentclass{amsart}
\usepackage{amsfonts}
\usepackage{graphicx}
\usepackage{tabularx}
\usepackage{array}
\usepackage[usenames,dvipsnames]{color}
\usepackage{comment}
\usepackage{amsmath}
\usepackage{amsthm}
\usepackage{amssymb}
\usepackage{fullpage}
\usepackage{xcolor}
\usepackage{tikz}

\tikzstyle{legend_general}=[rectangle, rounded corners, thin,
                          top color= white,bottom color=lavander!25,
                          minimum width=2.5cm, minimum height=0.8cm,
                          violet]

\DeclareMathOperator{\im}{im}

\newtheorem{theorem}{Theorem}[section]

\newtheorem{claim}[theorem]{Claim}
\newtheorem{lemma}[theorem]{Lemma}

\theoremstyle{definition}

\title{Transversals and bipancyclicity in bipartite graph families}

\author{Peter Bradshaw}
\address{Department of Mathematics, Simon Fraser University, Vancouver, Canada}
\email{pabradsh@sfu.ca}

\thanks{The author of this work has been partially supported by a supervisor's grant from the Natural Sciences and Engineering Research Council of Canada (NSERC)}

\begin{document}
\begin{abstract}
A bipartite graph is called \emph{bipancyclic} if it contains cycles of every even length from four up to the number of vertices in the graph.
A theorem of Schmeichel and Mitchem states that for $n \geq 4$, every balanced bipartite graph on $2n$ vertices in which each vertex in one color class has degree greater than $\frac{n}{2}$ and each vertex in the other color class has degree at least $\frac{n}{2}$ is bipancyclic. We prove a generalization of this theorem in the setting of graph transversals. Namely, we show that given a family $\mathcal{G}$ of $2n$ bipartite graphs on a common set $X$ of $2n$ vertices with a common balanced bipartition, if each graph of $\mathcal G$ has minimum degree greater than $\frac{n}{2}$ in one color class and minimum degree at least $\frac{n}{2}$ in the other color class, then there exists a cycle on $X$ of each even length $4 \leq \ell \leq 2n$ that uses at most one edge from each graph of $\mathcal G$. We also show that given a family $\mathcal G$ of $n$ bipartite graphs on a common set $X$ of $2n$ vertices meeting the same degree conditions, there exists a perfect matching on $X$ that uses exactly one edge from each graph of $\mathcal G$.
\end{abstract}
\maketitle
\section{Introduction}
One of the oldest problems in graph theory is that of determining conditions that guarantee a cycle of a given length in a graph. The classic Dirac's theorem \cite{Dirac} gives such a condition, stating that if every vertex of a graph $G$ on $n$ vertices has a degree of at least $\frac{n}{2}$, then $G$ contains a Hamiltonian cycle. A later result of Bondy \cite{Bondy} shows that the degree condition of Dirac's theorem guarantees more than just a Hamiltonian cycle; Bondy proves that under these same degree conditions, $G$ contains a cycle of every length $\ell$, for $3 \leq \ell \leq n$, except when $G \cong K_{n/2,n/2}$. A graph on $n$ vertices that contains a cycle of every length $\ell$, for $3 \leq \ell \leq n$, is called \emph{pancyclic}. Therefore, Bondy's result can be concisely summarized by the statement that every graph on $n$ vertices with minimum degree at least $\frac{n}{2}$ is pancyclic, except for the complete balanced bipartite graph.

The problem of finding conditions that guarantee cycles of given lengths in graphs has been considered not only for general graphs, but also specifically for bipartite graphs. Interestingly, the results of Dirac and Bondy mentioned above have close analogues in the bipartite setting. In particular, Moon and Moser \cite{Moon} prove the following theorem for bipartite graphs, which closely resembles Dirac's theorem. Here, we use $d(v)$ to denote the degree of a vertex $v$.
\begin{theorem}
\label{thmMoon}
Let $n \geq 2$. Let $G$ be a balanced bipartite graph on $2n$ vertices with a red and a blue color class. Suppose that for each red vertex $v \in V(G)$, $d(v) > \frac{n}{2}$. Suppose further that for each blue vertex $w \in V(G)$, $d(w) \geq \frac{n}{2}$. Then $G$ contains a Hamiltonian cycle.
\end{theorem}
Additionally, not only Dirac's theorem has a bipartite analogue, but Bondy's strengthened version of Dirac's theorem also has a close analogue for bipartite graphs. In particular, the following result of Schmeichel and Mitchem \cite{Schmeichel}, which generalizes Theorem \ref{thmMoon} for $n \geq 4$, closely resembles the result of Bondy mentioned above.
\begin{theorem}
\label{thmSchmeichel}
Let $n \geq 4$. Let $G$ be a balanced bipartite graph on $2n$ vertices with a red and a blue color class. Suppose that for each red vertex $v \in V(G)$, $d(v) > \frac{n}{2}$. Suppose further that for each blue vertex $w \in V(G)$, $d(w) \geq \frac{n}{2}$. Then $G$ contains a cycle of every even length $\ell$, for $4 \leq \ell \leq 2n$.
\end{theorem}

 A balanced bipartite graph on $2n$ vertices that contains a cycle of every even length $\ell$, for $4 \leq \ell \leq 2n$, is called \emph{bipancyclic}. Therefore, Theorem \ref{thmSchmeichel} may be summarized by stating that a balanced bipartite graph meeting certain minimum degree conditions is bipancyclic.
In Theorem \ref{thmSchmeichel}, the condition that $n \geq 4$ cannot be weakened, as when $n = 3$, the cycle on six vertices meets the minimum degree conditions of the theorem but contains no cycle of length four.
We note that the authors of \cite{Moon} and \cite{Schmeichel} also give stronger versions of Theorems \ref{thmMoon} and \ref{thmSchmeichel}, but we only consider the weaker versions stated here in order to highlight the fact that these results for bipartite graphs closely resemble the results of Dirac and Bondy for general graphs.

In this paper, we will also seek graph conditions that guarantee cycles of given lengths, but we will consider this question in the context of \emph{graph transversals}, a special type of \emph{transversals}, which are defined as follows. Suppose we have a universal set $\mathcal{E}$ and a family $\mathcal{F} = \{F_i: i \in I\}$ of subsets of $\mathcal{E}$, where $I$ is some index. Then a transversal on $\mathcal{F}$ is a set $\{f_i: i \in I\}$ such that $f_i \in F_i$ for each $i \in I$, and $f_i \neq f_j$ for $i \neq j$. In other words, a transversal is a system of distinct representatives for a family of sets. Transversals frequently appear in infinitary combinatorics under several similar definitions (see, for example, \cite{Erdos1968} and \cite{Aharoni1983}), and transversals are also extensively studied in the context of Latin squares (see \cite{Wanless} for a survey). 

The notion of a transversal may be applied to a family of graphs as follows. Suppose we have a set $X$ of $n$ vertices, and suppose that $\mathcal{G} = \{G_1, \dots, G_s\}$ is a family of $s$ graphs, each with $X$ as its set of vertices.
Then, we define a transversal on $\mathcal{G}$ to be a a set of $s$ edges $E \subseteq {X \choose 2}$ for which there exists a bijective function $\phi:E \rightarrow [s]$ such that for all $e \in E$, it holds that $e \in E(G_{\phi(e)})$. If $|E| < s$ and there exists an injective function $\phi:E \rightarrow [s]$ satisfying the same property, then we say that $E$ is a \emph{partial transversal} on $\mathcal{G}$. 
Informally, if we imagine that each graph $G_i \in \mathcal G$ has its edges colored with a single distinct color, then a transversal on $\mathcal G$ is a set of edges in which each color appears exactly once, and a partial transversal on $\mathcal G$ is a set of edges in which each color appears at most once.

Recently, certain classical results about cycles of specified lengths in graphs have been extended to the setting of graph transversals. For example, Aharoni et al.~\cite{Aharoni} give a transversal analogue of Mantel's theorem, showing that given a graph set $\mathcal{G} = \{G_1, G_2, G_3\}$ on a common set of $n$ vertices in which each graph $G_i$ has roughly at least $ 0.2557n^2$ edges, there must exist a partial transversal on $\mathcal{G}$ isomorphic to $K_3$. Additionally, Joos and Kim \cite{Joos} show that given a graph family $\mathcal{G} = \{G_1, \dots, G_n\}$ on a common vertex set of $n$ vertices, if the minimum degree of each graph $G_i$ is at least $\frac{n}{2}$, then $\mathcal{G}$ has a transversal isomorphic to a Hamiltonian cycle, which gives a generalization of Dirac's theorem. Cheng, Wang and Zhao \cite{Cheng} furthermore show that given a graph family $\mathcal{G} = \{G_1, \dots, G_n\}$ on a common vertex set $X$ of $n$ vertices, if the minimum degree of each graph $G_i$ is at least $ \frac{n+1}{2}$, then $\mathcal{G}$ contains a partial transversal of length $\ell$ for every length $3 \leq \ell \leq n-1$, which gives an approximate extension of Bondy's pancyclicity condition.

In this paper, we will focus on graph transversals over bipartite graph families. First, we will prove the following 
theorem, which gives a generalization of Theorem \ref{thmMoon} in the setting of graph transversals.

\begin{theorem}
\label{mainThm}
Let $n \geq 2$. Let $X$ be a set of $n$ red vertices and $n$ blue vertices, and let $\mathcal{G} = \{G_1, \dots, G_{2n}\}$ be a set of $2n$ bipartite graphs on the vertex set $X$. Suppose 
\begin{itemize}
\item For each $G_i \in \mathcal{G}$, the red vertices of $X$ and the blue vertices of $X$ both form independent sets in $G_i$;
\item For each $G_i \in \mathcal{G}$ and for each red vertex $v \in V(G_i)$, $d(v) > \frac{n}{2}$;
\item For each $G_i \in \mathcal{G}$ and for each blue vertex $w \in V(G_i)$, $d(w) \geq \frac{n}{2}$.
\end{itemize}
Then $\mathcal{G}$ contains a transversal isomorphic to a Hamiltonian cycle.
\end{theorem}
If we let $G_1 = G_2 = \dots = G_{2n}$ in Theorem \ref{mainThm}, then any Hamiltonian cycle in $G_1$ is a Hamiltonian transversal on $\mathcal G$, so Theorem \ref{mainThm} is a generalization of Theorem \ref{thmMoon}. After proving Theorem \ref{mainThm}, we will also be able to prove the following stronger theorem, which generalizes Theorem \ref{thmSchmeichel} in the setting of graph transversals.

\begin{theorem}
\label{thmPancyclic}
Let $n \geq 4$. Let $X$ be a set of $n$ red vertices and $n$ blue vertices, and let $\mathcal{G} = \{G_1, \dots, G_{2n}\}$ be a set of $2n$ bipartite graphs on the vertex set $X$. Suppose
\begin{itemize}
\item For each $G_i \in \mathcal{G}$, the red vertices of $X$ and the blue vertices of $X$ both form independent sets in $G_i$;
\item For each $G_i \in \mathcal{G}$ and for each red vertex $v \in V(G_i)$, $d(v) > \frac{n}{2}$;
\item For each $G_i \in \mathcal{G}$ and for each blue vertex $w \in V(G_i)$, $d(w) \geq \frac{n}{2}$.
\end{itemize}
Then $\mathcal{G}$ contains a (partial) transversal isomorphic to a cycle of length $\ell$ for each even length $\ell$, $4 \leq \ell \leq 2n$.
\end{theorem}

Again, if we let $G_1 = G_2 = \dots = G_{2n}$ in Theorem \ref{thmPancyclic}, then any cycle of length $\ell$ in $G_1$ is a partial transversal on $\mathcal G$, so Theorem \ref{thmPancyclic} is a generalization of Theorem \ref{thmSchmeichel}.
Theorem \ref{thmPancyclic} is also a generalization of Theorem \ref{mainThm} except for small values of $n$, and hence these results may be summarized by saying that if a graph family $\mathcal{G}$ satisfies the conditions of Theorems \ref{mainThm} and \ref{thmPancyclic}, then $\mathcal{G}$ is bipancyclic in the sense of graph transversals. The main tool used in our proofs is an auxiliary digraph technique introduced by Joos and Kim in \cite{Joos}.

We note that Theorems \ref{mainThm} and \ref{thmPancyclic} are best possible in that we may not relax the minimum degree condition. Indeed, for any even $n$, by letting each graph $G_i \in \mathcal{G}$ be an identical copy of a disjoint union of two complete graphs $K_{n/2,n/2}$, we find a graph family with minimum degree $\frac{n}{2}$ and no Hamiltonian transversal. A similar example shows that the minimum degree condition for odd $n$ may not be relaxed either.

Lastly, we will consider graph transversals isomorphic to perfect matchings, which we denote as \emph{perfect matching transversals}. Joos and Kim consider perfect matching transversals in \cite{Joos} and establish minimum degree conditions that guarantee perfect matching transversals in general graphs. Our last theorem shows that under the degree conditions of Theorems \ref{mainThm} and \ref{thmPancyclic}, a family of $n$ balanced bipartite graphs on a common set of $2n$ vertices contains a perfect matching transversal.
\begin{theorem}
\label{thmPM}
Let $n \geq 1$. Let $X$ be a set of $n$ red vertices and $n$ blue vertices, and let $\mathcal{G} = \{G_1, \dots, G_{n}\}$ be a set of $n$ bipartite graphs on the vertex set $X$. Suppose 
\begin{itemize}
\item For each $G_i \in \mathcal{G}$, the red vertices of $X$ and the blue vertices of $X$ both form independent sets in $G_i$;
\item For each $G_i \in \mathcal{G}$ and for each red vertex $v \in V(G_i)$, $d(v) > \frac{n}{2}$;
\item For each $G_i \in \mathcal{G}$ and for each blue vertex $w \in V(G_i)$, $d(w) \geq \frac{n}{2}$.
\end{itemize}
Then $\mathcal{G}$ contains a perfect matching transversal.
\end{theorem}

Theorem \ref{thmPM} is equivalent to a special case of a matching theorem for $r$-partite $r$-uniform hypergraphs by Aharoni, Georgakopoulos, and Spr\"{u}ssel in \cite{AharoniRGraphs}, for the case $r = 3$. We nevertheless include a proof of Theorem \ref{thmPM} in order to demonstrate the power of Joos and Kim's auxiliary digraph from \cite{Joos} technique for investigating graph transversals. We will also give a construction showing that the degree condition of Theorem \ref{thmPM} is, in a sense, best possible.

\section{A Hamiltonian transversal: Proof of Theorem \ref{mainThm}}
This section will be dedicated to proving Theorem \ref{mainThm}. When $n = 2$, the theorem is trivial; thus we will assume throughout the proof that $n \geq 3$. We will let $X$ and $\mathcal{G}$ be defined as in Theorem \ref{mainThm}. We define a \emph{Hamiltonian transversal} on $\mathcal G$ as a transversal on $\mathcal G$ isomorphic to a Hamiltonian cycle. 

Throughout this section, we will let $X$ have a blue vertex set $\{p_1, \dots, p_n\}$ and a red vertex set $\{q_1, \dots, q_n\}$. We will assume throughout this section that $\mathcal G$ does not contain a Hamiltonian transversal, and we will arrive at a contradiction. This strategy will allow us to write certain steps of the proof more concisely. The first goal in our proof will be to establish the following claim.
\begin{claim}
\label{claim2n2}
$\mathcal{G}$ contains a partial transversal isomorphic to the disjoint union of a cycle of length $2n - 2$ and a $K_2$.
\end{claim}
In order to prove Claim \ref{claim2n2}, we will first prove two auxiliary claims. The first of these auxiliary claims shows that Claim \ref{claim2n2} holds whenever $\mathcal G$ contains a partial transversal isomorphic to a Hamiltonian path.

\begin{claim}
If $\mathcal{G}$ contains a partial transversal isomorphic to a Hamiltonian path, then $\mathcal{G}$ contains a partial transversal isomorphic to the disjoint union of a cycle of length $2n - 2$ and a $K_2$.
\label{lemmaHamPath}
\end{claim}
\begin{proof}
Let $P$ be a partial transversal in $\mathcal{G}$ isomorphic to a Hamiltonian path. We assume without loss of generality that $P$ has a vertex sequence $(q_1, p_2, q_2, p_3, \dots,p_n, q_n, p_1)$. We let $P$ have an associated injective function $\phi:E(P) \rightarrow [2n]$. We write $\phi(p_2 q_2) = m_1$ and $[2n] \setminus \im(\phi) = \{m_2\}$.
We show the key parts of $\mathcal G$ in Figure \ref{fig2n2}.
 If $p_1 q_{1} \in E(G_{m_2})$, then $\mathcal G$ contains a Hamiltonian transversal, and if $p_1 q_2  \in E(G_{m_2})$, then the claim is proven; hence, we may assume that no such edge exists in $G_{m_2}$. (Note that this immediately proves the claim when $n = 3$, and hence we may assume that $n \geq 4$.) We may similarly assume that $ p_1 q_2 \not \in E(G_{m_1})$. 

By our minimum degree conditions on red vertices and the assumption that $p_1 q_2 \not \in E(G_{m_1})$, there must exist at least $\frac{n}{2} - \frac{3}{2}$ edges of the form $p_j q_2 \in E(G_{m_1})$ with $j \in [4,n]$. Similarly, by our minimum degree conditions on blue vertices and our assumption that $p_1 q_1, p_1 q_2 \not \in E(G_{m_2})$, there must exist at least $\frac{n}{2} - 1$ values $j \in [4,n]$ for which $ p_1 q_{j-1}  \in E(G_{m_2})$. As there exist a total of $n - 3$ values $j \in [4,n]$, it follows by the pigeonhole principle that there exists a value $j \in [4,n]$ for which $p_j  q_2 \in E(G_{m_1})$ and $ p_1 q_{j-1} \in E(G_{m_2})$. Thus, $\mathcal G$ contains a partial transversal $C$ isomorphic to a cycle of length $2n - 2$ with a vertex sequence 
$$(q_2, p_j, q_j, p_{j+1}, q_{j+1}, \dots, p_n, q_n, p_1, q_{j-1}, p_{j-1}, \dots, q_3, p_3, q_2),$$
as shown in Figure \ref{fig2n2},
 and with an injective function $\psi:C \rightarrow [2n]$ satisfying $[2n] \setminus \im (\psi) = \{\phi( q_1   p_2 ), \phi(q_{j-1}p_j)\}$.
 Therefore, $C \cup \{ q_1  p_2 \}$ gives a partial transversal isomorphic to the disjoint union of a cycle of length $2n - 2$ and a $K_2$, and the claim is proven.
\end{proof}

\begin{figure}
\begin{tikzpicture}
[scale=2.2,auto=left,every node/.style={circle,fill=gray!30,minimum size = 6pt,inner sep=0pt}]
\node (z) at (-0.6,0.95) [fill = white]  {$q_n$};
\node (z) at (0.98,0.58) [fill = white]  {$q_2$};
\node (z) at (-0.9,0.66) [fill = white]  {$p_n$};
\node (z) at (0.67,0.96) [fill = white]  {$p_2$};
\node (z) at (0.66,0.58) [fill = white]  {$m_1$};
\node (z) at (0.28,-1.15) [fill = white]  {$q_{j-1}$};
\node (z) at (-0.18,-1.15) [fill = white]  {$p_j$};
\draw (0,0) circle [thick, radius=1];
\node (z) at (0,1) [fill = white] {};
\node (z) at (-0.15,1) [fill = white] {};
\node (z) at (0.2,1) [fill = white] {};
\node (z) at (-0.07,1) [fill = white] {};
\node (z) at (0.1,1) [fill = white] {};
%\node (z) at (-0.15,1) [fill = white] {};
\node (z) at (0.15,1) [fill = white] {};
\node (z) at (-0.07,1) [fill = white] {};
\node (z) at (0.07,1) [fill = white] {};
%\node (z) at (-0.33,-0.75) [fill = white]  {$j-1$};
%\node (z) at (0.3,0.75) [fill = white]  {$1$};
%\node (z) at (-0.3,0.75) [fill = white]  {$n$};
%\draw[yellow] [line width=0.5mm] (-0.25,0.97) -- (-0.18,0.6);
%\draw[blue] [line width=0.5mm] (0.25,0.97) -- (0.2,0.6);
%\node (z) at (0,1.15) [fill = white]  {$x_1$};

\node (z) at (0.28,1.12) [fill = white]  {$q_1$};
\node (z) at (-0.18,1.12) [fill = white]  {$p_1$};

\node (p2) at (0.6,0.8) [draw = black, fill = blue!30] {};
\node (qn) at (-0.6,0.8) [draw = black, fill = red!30] {};
\node (pn) at (-0.86,0.5) [draw = black, fill = blue!30]  {};

\node (q1) at (0.25,0.97) [draw = black, fill = red!30] {};
\node (p1) at (-0.18,0.98) [draw = black, fill = blue!30] {};
\node (qj1) at (0.25,-0.97) [draw = black, fill = red!30] {};
\node (pj) at (-0.18,-0.98) [draw = black, fill = blue!30] {};

%\node (z) at (1.1,0.3) [fill = white]  {$x_i$};

\node (z) at (0.58,-0.2) [fill = white]  {$m_1$};
\node (z) at (-0.15,0.17) [fill = white]  {$m_2$};

%\draw[yellow] [line width=0.5mm] (0.86,0.5) -- (0.5,0.3);
\node (q2) at (0.86,0.5) [draw = black, fill = red!30]  {};
%\draw[blue] [line width=0.5mm] (0.997,0.07) -- (0.6,0.01);
%\node (p3) at (0.997,0.07) [draw = black, fill = blue!30]  {};

%\draw[blue] [line width=0.5mm] (0.752,-0.652) -- (0.45,-0.4);
%\node (z) at (0.752,-0.652) [draw = black]  {};

%\draw[yellow] [line width=0.5mm] (0.95,-0.32) -- (0.58,-0.17);
%\node (z) at (0.95,-0.32) [draw = black]  {};

%\draw[yellow] [line width=0.5mm] (-0.95,0.3) -- (-0.6,0.17);
%\node (z) at (-0.95,0.3) [draw = black]  {};

%\draw[blue] [line width=0.5mm] (-0.75,0.66) -- (-0.47,0.43);
%\node (z) at (-0.75,0.66) [draw = black]  {};

%\filldraw (0,1) circle (1.8pt);
%\filldraw (0.955,0.3) circle (1.8pt);
%\filldraw (0.866,-0.5) circle (1.8pt);
%\filldraw (-0.7071,-0.7071) circle (1.8pt);
%\filldraw (-0.866,0.5) circle (1.8pt);
 \draw[dashed] (qj1) edge  (p1);
 \draw[dashed] (pj) edge  (q2);

\end{tikzpicture} 
\caption{The figure shows the key parts of the graph family $\mathcal G$ considered in Claim \ref{lemmaHamPath}. The Hamiltonian path partial transversal $P$ is represented by the broken circle with endpoints $q_1$ and $p_1$. The edge $p_2 q_2$ belongs to $G_{m_1}$, and $P$ ``misses" the graph $G_{m_2}$. If there exist two edges belonging to $G_{m_1}$ and $G_{m_2}$ in the form of the dotted edges, then $\mathcal G$ must contain a partial transversal containing the single edge $q_1 p_2$ as well as a cycle of length $2n - 2$ containing all other vertices of $X$.}
\label{fig2n2}
\end{figure}
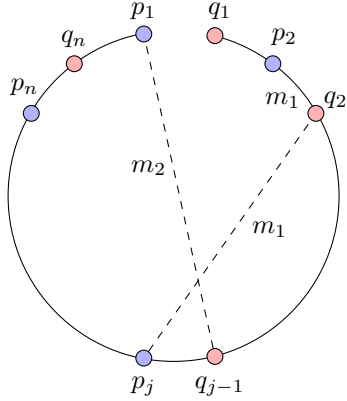 

In the next claim, we show that $\mathcal{G}$ must contain a partial transversal isomorphic to a cycle of length $2n - 2$. The proof of the claim closely follows a method of Moon and Moser \cite{Moon}, which in turn closely follows a method of P\'{o}sa \cite{Posa1962}.

\begin{claim}
$\mathcal{G}$ contains a partial transversal isomorphic to a cycle of length $2n-2$.
\label{lemmaCycle}
\end{claim}
\begin{proof}
Suppose the claim does not hold. Let $\mathcal{G}$ be an edge-maximal counterexample to the claim; that is, let $\mathcal{G}$ be a family of graphs such that after adding any edge to any graph $G_i \in \mathcal{G}$, the resulting family contains either a Hamiltonian transversal or a partial transversal isomorphic to a cycle of length $2n-2$. If each graph $G_i \in \mathcal{G}$ is a complete bipartite graph, then $\mathcal{G}$ certainly contains a Hamiltonian transversal. Therefore, we may assume without loss of generality that $G_1$ is not a complete bipartite graph and that we may add some edge $e$ to $G_1$ with a red endpoint and a blue endpoint. We may further assume that $\mathcal{G} + e = \{G_1 + e, G_2, \dots, G_{2n}\}$ contains a partial transversal $C$ isomorphic to a cycle of length $2n - 2$, since if $\mathcal G + e$ contains a Hamiltonian transversal, then $\mathcal G$ must contain a partial transversal isomorphic to a Hamiltonian path, and we are done by Claim \ref{claim2n2}.

%We claim that we may assume that $\mathcal{G} + e = \{G_1 + e, G_2, \dots, G_{2n}\}$ contains a partial transversal isomorphic to a cycle of length $2n - 2$. Indeed, if $\mathcal G + e$ contains a Hamiltonian transversal, then $\mathcal G$ must contain a partial transversal isomorphic to a Hamiltonian path, and then the claim follows from Claim \ref{claim2n2}.

% To prove this claim, we will suppose that $\mathcal{G} + e$ contains a Hamiltonian transversal and no partial transversal isomorphic to a cycle of length $2n-2$. We will show that in this case, we may choose a different edge $e'$ to add to $G_1$ such that $\mathcal{G} + e' = \{G_1 + e', G_2, \dots, G_{2n}\}$ contains a partial transversal isomorphic to a cycle of length $2n - 2$.

%We assume without loss of generality that $e = p_1 q_1$ and $\mathcal{G} + e$ contains a Hamiltonian transversal with a vertex sequence $( p_1, q_1, p_2, q_2, p_3, \dots, p_n, q_n, p_1)$. If $p_1 q_2 \in E(G_1)$, then $\mathcal{G}$ already contains a partial transversal isomorphic to a cycle of length $2n -2$, and the claim is proven. Otherwise, $p_1 q_2 \not \in E(G_1)$, and we could have chosen the edge $e' = p_1 q_2$ to let $\mathcal{G} + e'$ contain a partial transversal isomorphic to a cycle of length $2n - 2$. Hence, we assume that $\mathcal{G} + e$ contains a partial transversal $C$ isomorphic to a cycle of length $2n -2$.

We say without loss of generality that $e = p_1 q_1$. We also say without loss of generality that $C$ has a vertex sequence $(p_1, q_1, p_2, q_2, p_3, \dots, p_{n-1}, q_{n -1}, p_1)$. Let $\phi:E(C) \rightarrow [2n]$ be the injective function associated with the partial transversal $C$, and let $[2n] \setminus \im(\phi) = \{m_1, m_2\}$.
By our minimum degree conditions on the red vertices and the assumption that $p_1 q_1 \not \in E(G_{1})$, there must exist at least $\left \lceil \frac{n}{2} - \frac{1}{2} \right \rceil$ values $j \in [2,n-1]$ for which $q_1 p_j \in E(G_1)$. Similarly, by our minimum degree conditions on the blue vertices, there must exist at least $\left \lceil \frac{n}{2} - 2 \right \rceil $ values $j \in [2,n-1]$ for which $p_1 q_{j-1} \in E(G_{m_1})$. Then, one of two cases must occur, and we will see that in both cases, our minimal counterexample $\mathcal G$ is not actually a counterexample to the claim, which will complete the proof.
\begin{enumerate}
\item There exists a value $j \in [2,n-1]$ for which $q_1 p_j \in E(G_1)$ and $p_1 q_{j-1} \in E(G_{m_1})$. Then, $\mathcal G$ has a partial transversal isomorphic to a cycle of length $2n - 2$ with a vertex sequence of $$(q_1, p_2, q_2, \dots, q_{j-1}, p_1, q_{n-1}, p_{n-1}, \dots, p_j, q_1),$$
and $\mathcal G$ is not a counterexample to the claim.
\item Otherwise, as $\left \lceil \frac{n}{2} - \frac{1}{2} \right \rceil + \left \lceil \frac{n}{2} - 2 \right \rceil = n-2$, it must follow from the pigeonhole principle that there exist exactly $\left \lceil \frac{n}{2} - \frac{1}{2} \right \rceil$ values $j \in [2,n-1]$ for which $p_1 q_j \in E(G_1)$ and exactly $\left \lceil \frac{n}{2} - 2 \right \rceil $ values $j \in [2,n-1]$ for which $q_1 p_{j-1} \in E(G_{m_1})$. Then, it must follow that $p_1 q_n \in E(G_1)$ and  $q_1 p_n  \in E(G_{m_1})$. Then, $\mathcal{G}$ contains a partial transversal isomorphic to a Hamiltonian path, namely one beginning at $p_n$ and ending at $q_n$, and by Claim \ref{lemmaHamPath}, $\mathcal G$ is not a counterexample to the claim.
\end{enumerate}
%As $\mathcal{G}$ contains no partial transversal isomorphic to a cycle of length $2n -2$, for each edge $p_1 q_j \in E(G_1)$, $2 \leq j \leq n-1$, we must have $q_1 p_{j-1} \not \in E(G_{m_1})$; otherwise,  gives a vertex sequence for a partial transversal isomorphic to a cycle of length $2n - 2$ in $\mathcal{G}$. If either $p_1 q_n \not \in E(G_1)$ or $q_1 p_n \not \in E(G_{m_1})$, then it follows that $d_{G_1}(p_1) + d_{G_{m_1}}(q_1) \leq n$, a contradiction to the degree requirements of the graphs in $\mathcal{G}$. Therefore, we may assume that $p_1 q_n \in E(G_1), q_1 p_n  \in E(G_{m_1})$. It then follows that $\mathcal{G}$ contains a partial transversal isomorphic to a Hamiltonian path, namely one beginning at $p_n$ and ending at $q_n$. Then, by Claim \ref{lemmaHamPath}, $\mathcal{G}$ contains either a Hamiltonian transversal or a partial transversal isomorphic to a cycle of length $2n - 2$.
\end{proof}

We are now ready to prove Claim \ref{claim2n2}, after which we will be ready for the main method of our proof. \\

\textit{Proof of Claim \ref{claim2n2}:}
By Claim \ref{lemmaCycle}, we may assume that $\mathcal{G}$ contains a partial transversal $C$ isomorphic to a cycle of length $2n - 2$. Let $\phi:E(C) \rightarrow [2n]$ be the injective function associated with $C$, and let $[2n] \setminus \im(\phi) = \{m_1, m_2\}$. Let $C$ have a vertex sequence $(p_1, q_1, p_2, q_2, \dots, p_{n-1}, q_{n-1}, p_1)$. 
%Then $X \setminus V(C) = \{p_n,q_n\}$. 
If $p_nq_n \in E(G_{m_1})$ or $p_nq_n \in E(G_{m_2})$, then the claim is proven; otherwise, for each edge $e\in E(G_{m_1}) \cup E(G_{m_2})$, $e$ has an endpoint in $V(C)$.

We aim to show that $\mathcal{G}$ contains a partial transversal $P$ isomorphic to a Hamiltonian path. Let $A \subseteq V(C)$ be the set of vertices of $V(C)$ adjacent to $q_n$ via $G_{m_1}$. Recall that $A$ is a set of blue vertices. As $|A| \geq \frac{n}{2} + \frac{1}{2}$, it follows that at most $\frac{n}{2} - \frac{3}{2}$ red vertices of $C$ are not adjacent in $C$ to a vertex of $A$. Therefore, as $p_n$ has at least $\frac{n}{2} $ red neighbors in $V(C)$ via $G_{m_2}$, there must exist a blue vertex $a \in A$, a red vertex $b \in V(C)$ adjacent to $a$ in $C$, and an edge $bp_n \in E(G_{m_2})$. Then $\mathcal G$ contains a partial transversal isomorphic to a Hamiltonian path beginning at $p_n$ and ending at $q_n$. Then, by Claim \ref{lemmaHamPath}, we may find a partial transversal in $\mathcal{G}$ isomorphic to the disjoint union of a cycle of length $2n - 2$ and a $K_2$.
\qed \\

Now, with Claim \ref{claim2n2} in place, we are ready for the main idea of Theorem \ref{mainThm}. We will follow a method of Joos and Kim \cite{Joos} used for proving a transversal version of Dirac's theorem. The method of Joos and Kim fits our proof very closely, and therefore the remainder of our proof uses the ideas of \cite{Joos} with very few changes.

We let $\mathcal G$ have a partial transversal isomorphic to the disjoint union of a cycle $C$ of length $2n - 2$ and a graph $K \cong K_2$. We will rename our vertices; we say that $V(K) = \{x,y\}$, and we let $C$ have a vertex sequence $(v_1, v_2, \dots, v_{2n-2}, v_1)$. We let $x, v_2, v_4, \dots, v_{2n-2}$ be red vertices, and we let $y, v_1, v_3, \dots, v_{2n-3}$ be blue vertices. For each vertex $v_i \in V(C)$, we say that $\phi(v_i v_{i+1}) = i$ (where $v_{2n - 1}$ is identified with $v_1$). We let $\phi(xy) = 2n - 1$. With these assignments, $\phi$ ``misses" the value $2n$.

We define an auxiliary digraph $H$ on $X$. For every red vertex $v_i \in V(C)$ and for every edge $v_i v_j \in E(G_i)$ with $ j \neq i + 1$, we let $H$ include the arc $v_i v_j$. Furthermore, if $v_i y \in E(G_i)$, we let $H$ contain the arc $v_i y$. We write $d^+(v)$ and $d^-(v)$ respectively for the out-degree and in-degree of a vertex $v \in V(H)$. Note that for each red vertex $v_i \in V(H)$, $v_i$ has at least $\frac{n}{2} + \frac{1}{2}$ incident edges in the graph $G_i$, and hence $d^+(v_i) \geq \frac{n}{2} - \frac{1}{2}$.

\begin{claim}
\label{claimY}
$d^-(y) \leq \frac{n}{2} - 1$.
\end{claim}
\begin{proof}
%First, we argue that $d^-(y) \leq \frac{n}{2} - \frac{1}{2}$. As $\mathcal G$ has no Hamiltonian transversal, for each in-neighbor $v_i$ of $y$, we must have $v_{i+1}x \not \in E(G_{2n})$. Thus, if $y$ has at least $\frac{n}{2} $ in-neighbors $v_i$, there hence exist at least $\frac{n}{2}$ vertices $v_j$ for which $v_j x \not \in E(G_{2n})$. Therefore, the set $N_{G_{2n}}(x) \cap V(C)$ is of order at most $(n - 1) - \frac{n}{2} = \frac{n}{2} - 1$. This is a contradiction, as $d_{G_{2n}}(x) \geq \frac{n}{2}+\frac{1}{2}$, and thus $x$ has at least $\frac{n}{2} - \frac{1}{2}$ neighbors in $V(C)$ via $G_{2n}$. Therefore, we conclude that $d^-(y) \leq \frac{n}{2} - \frac{1}{2}$. 
Suppose $d^-(y) \geq \frac{n}{2} - \frac{1}{2}$. Let $v_j \in V(C)$ be a neighbor of $x$ via $G_{2n}$. For every in-neighbor $v_i$ of $y$, we must have $v_{i+1} v_{j+1} \not \in E(G_j)$, as otherwise, $\mathcal G$ contains a Hamiltonian transversal with the vertex sequence
$$(v_j,x,y,v_i,v_{i-1}, \dots, v_{j+2}, v_{j+1}, v_{i+1}, v_{i+2}, \dots, v_j),$$
as shown in Figure \ref{figIndegClaim}. Therefore,  
$$|N_{G_j}(v_{j+1}) \cap V(C)| \leq (n-1) - \left ( \frac{n}{2} - \frac{1}{2} \right ) = \frac{n}{2} - \frac{1}{2},$$
and hence we must have $v_{j+1} y \in E(G_j)$. Then, $\mathcal G$ contains a Hamiltonian transversal with a vertex sequence $(v_1, \dots, v_j, x, y, v_{j+1}, v_{j+2}, \dots, v_{2n-2}, v_1)$, a contradiction. Therefore, we conclude that $d^-(y) \leq \frac{n}{2} - 1$.
\end{proof}

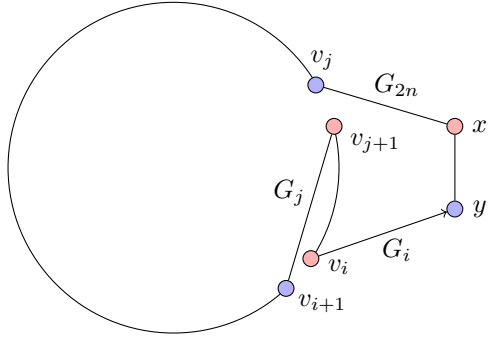
\begin{figure}
\begin{tikzpicture}
[scale=2.2,auto=left,every node/.style={circle,fill=gray!30,minimum size = 6pt,inner sep=0pt}]
\node (z) at (0.9,-0.82) [fill = white]  {$v_{i+1}$};
\node (z) at (0.9,0.65) [fill = white]  {$v_j$};
\node (z) at (1,-0.6) [fill = white]  {$v_{i}$};
\node (z) at (1.22,0.15) [fill = white]  {$v_{j+1}$};
\draw (0,0) circle [thick, radius=1];
%\node (z) at (0,1) [fill = white] {};
%\node (z) at (-0.2,1) [fill = white] {};
%\node (z) at (0.2,1) [fill = white] {};
%\node (z) at (-0.1,1) [fill = white] {};
%\node (z) at (0.1,1) [fill = white] {};
%\node (z) at (-0.15,1) [fill = white] {};
%\node (z) at (0.15,1) [fill = white] {};
%\node (z) at (-0.07,1) [fill = white] {};
%\node (z) at (0.07,1) [fill = white] {};

;
\node (z) at (0.86,0.5) [fill = white] {};
\node (z) at (0.9,0.45) [fill = white] {};
\node (z) at (0.92,0.4) [fill = white] {};
\node (z) at (0.93,0.35) [fill = white] {};
\node (z) at (0.93,0.3) [fill = white] {};

\node (z) at (0.83,-0.55) [fill = white] {};
\node (z) at (0.8,-0.6) [fill = white] {};
\node (z) at (0.75,-0.65) [fill = white] {};
\node (z) at (0.7,-0.7) [fill = white] {};
%\node (z) at (-0.33,-0.75) [fill = white]  {$j-1$};
%\node (z) at (0.3,0.75) [fill = white]  {$1$};
%\node (z) at (-0.3,0.75) [fill = white]  {$n$};
%\draw[yellow] [line width=0.5mm] (-0.25,0.97) -- (-0.18,0.6);
%\draw[blue] [line width=0.5mm] (0.25,0.97) -- (0.2,0.6);
%\node (z) at (0,1.15) [fill = white]  {$x_1$};

%\node (z) at (0.28,1.12) [fill = white]  {$v_2$};
%\node (z) at (-0.28,1.12) [fill = white]  {$v_1$};

%\node (v2) at (0.25,0.97) [draw = black, fill = red!30] {};
%\node (v1) at (-0.25,0.97) [draw = black, fill = blue!30] {};

%\node (z) at (1.1,0.3) [fill = white]  {$x_i$};

%\node (z) at (-0.36,-0.5) [fill = white]  {$G_k$};
\node (z) at (0.7,-0.15) [fill = white]  {$G_j$};

\node (z) at (1.35,0.5) [fill = white]  {$G_{2n}$};
\node (z) at (1.35,-0.5) [fill = white]  {$G_{i}$};
%\node (z) at (1.55,0.04) [fill = white]  {$G_{2n}$};
%\draw[yellow] [line width=0.5mm] (0.86,0.5) -- (0.5,0.3);
\node (vj) at (0.86,0.5) [draw = black, fill = blue!30]  {};
%\draw[blue] [line width=0.5mm] (0.997,0.07) -- (0.6,0.01);
\node (vj1) at (0.97,0.25) [draw = black, fill = red!30]  {};

\node (z) at (1.85,0.25) [fill = white]  {$x$};
\node (z) at (1.85,-0.25) [fill = white]  {$y$};
%\node (x) at (1.9,0.5) [draw = black, fill = red!30]{};
%\node (y) at (1.95,0.25) [draw = black, fill = blue!30]  {};

\node (x) at (1.7,0.25) [draw = black, fill = red!30]{};
\node (y) at (1.7,-0.25) [draw = black, fill = blue!30]  {};

%\draw[blue] [line width=0.5mm] (0.752,-0.652) -- (0.45,-0.4);
%\node (z) at (0.752,-0.652) [draw = black]  {};

%\draw[yellow] [line width=0.5mm] (0.95,-0.32) -- (0.58,-0.17);
%\node (z) at (0.95,-0.32) [draw = black]  {};

%\draw[yellow] [line width=0.5mm] (-0.55,-0.83) -- (-0.36,-0.51);
\node (vk1) at (0.68,-0.73) [draw = black, fill = blue!30]  {};
%\draw[blue] [line width=0.5mm] (-0.83,-0.55) -- (-0.55,-0.35);
\node (vk) at (0.83,-0.55) [draw = black, fill = red!30]  {};

%\draw[yellow] [line width=0.5mm] (-0.95,0.3) -- (-0.6,0.17);
%\node (z) at (-0.95,0.3) [draw = black]  {};

%\draw[blue] [line width=0.5mm] (-0.75,0.66) -- (-0.47,0.43);
%\node (z) at (-0.75,0.66) [draw = black]  {};

%\filldraw (0,1) circle (1.8pt);
%\filldraw (0.955,0.3) circle (1.8pt);
%\filldraw (0.866,-0.5) circle (1.8pt);
%\filldraw (-0.7071,-0.7071) circle (1.8pt);
%\filldraw (-0.866,0.5) circle (1.8pt);
 \path[->] (vk) edge  (y);
% \draw (c) edge [ultra thick] (v2);
\foreach \from/\to in {x/vj, x/y, vj1/vk1}
    \draw (\from) -- (\to);

\end{tikzpicture} 
\caption{In the figure, the broken circle represents the edges of the partial transversal $C$ considered in the proof of Theorem \ref{mainThm}.
The figure shows one form of Hamiltonian transversal that may be found in $\mathcal G$ using the argument in Claim \ref{claimY} if $d^-(y) > \frac{n}{2} - 1$. }
\label{figIndegClaim}
\end{figure}

Next, we show that some blue vertex $w \in V(C)$ satisfies $d^{-}(w) \geq \frac{n}{2} - \frac{1}{2}$. Recall that each red vertex of $V(C)$ has out-degree at least $\frac{n}{2} - \frac{1}{2}$. Let $E^-(C)$ denote the set of arcs in $H$ that end in $V(C)$. As $d^-(y) \leq \frac{n}{2} -  1$, we have that 
\begin{equation}
\label{eqnEC} \tag{$\star$}
|E^-(C)| \geq \left (\frac{n}{2} - \frac{1}{2} \right )(n - 1) - \left ( \frac{n}{2} - 1 \right ) > (n - 1) \left (\frac{n}{2} - \frac{3}{2} \right ). 
\end{equation}
Therefore, by the pigeonhole principle, some blue vertex $w \in V(C)$ must satisfy $d^-(w) \geq \frac{n}{2} - 1$. Furthermore, if all blue vertices $w' \in V(C)$ satisfy $d^-(w') \leq \frac{n}{2} - 1$, then $d^-(w^*) = \frac{n}{2} - 1$ for some blue vertex $w^*$, and hence $n$ is even. However, in this case, each red vertex of $V(C)$ has an out-degree of at least $\frac{n}{2}$, and we can obtain a better lower bound for $|E^-(C)|$:
$$|E^-(C)| \geq \frac{n}{2} (n - 1) - \left ( \frac{n}{2} - 1 \right ) > (n - 1) \left (\frac{n}{2} - \frac{1}{2} \right).$$
This second inequality implies that some blue vertex of $V(C)$ has in-degree at least $\frac{n}{2}$. Therefore, we may proceed with the assumption that there exists a blue vertex $w \in V(C)$ for which $d^-(w) \geq \frac{n}{2} - \frac{1}{2}$. 

Now, we consider two cases. In both cases, we will find a Hamiltonian transversal on $\mathcal G$, which will complete the proof. \\ \\
\textbf{Case 1:} There exists a vertex $w \in V(C)$ for which $d^-(w) \geq \frac{n}{2}$.

We will attempt to find a partial transversal on $\mathcal G$ isomorphic to a Hamiltonian path and satisfying certain conditions.
Without loss of generality, we assume $d^-(v_1) \geq \frac{n}{2} $. We claim that the set 
$$N^*:=(N_{G_1}(x) \cup N_{G_{2n}}(y)) \cap V(C)$$ 
is not an independent set in the cycle $C$. Indeed, suppose that $N^*$ is an independent set in $C$. The set $N_{G_1}(x) \cap V(C)$ must contain at least $\frac{n}{2} - \frac{1}{2}$ vertices, and thus there must exist at least $\frac{n}{2} + \frac{1}{2}$ vertices adjacent in $C$ to $N_{G_1}(x) \cap V(C)$. Then, if $N^*$ is an independent set, $N_{G_{2n}}(y) \cap V(C)$ contains at most 
$(n-1) - \left (\frac{n}{2} + \frac{1}{2} \right) < \frac{n}{2} - 1$
vertices, a contradiction. Therefore, there exists a pair of vertices $v_j, v_{j+1}$ such that either $v_j \in N_{G_1}(x)$ and $v_{j+1} \in N_{G_{2n}}(y)$, or $v_{j+1} \in N_{G_1}(x)$ and $v_{j} \in N_{G_{2n}}(y)$
(where $v_{2n - 1}$ is identified with $v_1$). 
If $j = 1$, then $\mathcal{G}$ contains a Hamiltonian transversal, which may be observed in Figure \ref{fig1} by identifying $(v_j, v_{j+1})$ with $(v_1, v_2)$. Otherwise, we have a partial transversal $P$ isomorphic to a Hamiltonian path with an injective function $\psi$ and a vertex sequence of the form 
$$
 (v_2, v_3, \dots, v_j, \sigma, \tau, v_{j+1}, \dots, v_{2n-2}, v_1),
$$
where either $(\sigma, \tau) = (x,y)$ or $(\sigma, \tau) = (y,x)$.
 We rename this vertex sequence of $P$ as $(v^1, v^2, \dots, v^{2n})$. We observe that $[2n] \setminus \im(\psi) = \{j\}$. 

Finally, we consider the following sets, again identifying $v_{2n - 1}$ with $v_1$:
$$I_j:= \{i \in [2n - 2]: v_{i+1} v_2 \in E(G_j)\},$$
$$I^-:= \{i \in [2n - 2]: v_i \in N^-(v_1)\}.$$
As $d^-(v_1) \geq \frac{n}{2}$ and $d^+(x) = 0$, we know that $|I^-| \geq \frac{n}{2}$, and as $d_{G_j}(v_2) \geq \frac{n}{2} + \frac{1}{2}$, we know that $|I_j| \geq \frac{n}{2} - \frac{1}{2}$. 
Furthermore, if $v_1 v_2 \in E(G_j)$, then $\mathcal{G}$ contains a Hamiltonian transversal; hence we may assume that $I_j \subseteq [2n - 4] \cap 2 \mathbb{Z}$. Similarly, as $v_1$ does not have an in-neighbor $v_{2n - 2}$, we may also assume that $I^- \subseteq [2n - 4] \cap 2 \mathbb{Z}$. Therefore, $I_j$, $I^-$ are both subsets of $[2n - 4] \cap 2 \mathbb{Z}$, which contains $n-2$ elements. Hence, as $|I_j|+ |I^-| = n$, we have at least two elements in the intersection $ I_j \cap I^-$, and in particular we have an element $k \in I_j \cap I^-$ for which $ k \neq j$. We may write $v^l = v_k$, and as $k \neq j$, $v_{k+1} = v^{l+1}$. Then, we see that 
$$(v^1, v^2, \dots, v^l, v^{2n-2}, v^{2n-1}, \dots, v^{l+1}, v^1)$$
 gives a vertex sequence for a Hamiltonian transversal on $\mathcal{G}$, which we illustrate in Figure \ref{fig1} (i). Thus, in this case, the proof is complete. \\ \\
\textbf{Case 2:} Every vertex $w \in V(C)$ satisfies $d^{-}(w) \leq \frac{n}{2} - \frac{1}{2}$.

In this case, as there must exist a vertex in $V(C)$ with an in-degree of exactly $\frac{n}{2}- \frac{1}{2}$, we know that $n$ is odd.
 Let $a$ be the number of vertices $w \in V(C)$ satisfying $d^-(w) = \frac{n}{2} - \frac{1}{2}$. Then, by $(\star)$, we have
$$a \left ( \frac{n}{2} - \frac{1}{2} \right ) + (n-1-a) \left ( \frac{n}{2} - \frac{3}{2} \right ) \geq \left (\frac{n}{2} - \frac{1}{2} \right )(n - 1) - \left ( \frac{n}{2} - 1 \right ),$$
or equivalently, $a \geq \frac{n}{2}$. 
%However, as $a$ is an integer, we may actually conclude that $a \geq \frac{n}{2} + \frac{1}{2}$. 
Therefore, since $|N_{G_{2n}}(y) \cap V(C) | \geq \frac{n}{2} - \frac{1}{2}$, we may choose a neighbor $v_{j+1}$ of $y$ via $G_{2n}$ such that $d^-(v_{j}) = \frac{n}{2} - \frac{1}{2}$. Note that since $v_j$ is a blue vertex, $j$ is odd.

Finally, we consider the following sets, again identifying $v_{2n - 1}$ with $v_1$:
$$I_j:= \{i \in [2n - 2]: v_{i+1} x \in E(G_j)\},$$
$$I^-:= \{i \in [2n - 2]: v_{i} \in N^-(v_j)\}.$$
If $j-1 \in I_j$, then $\mathcal G$ must contain a Hamiltonian transversal, so we may assume that $j-1 \not \in I_j$. Furthermore, by construction of $H$, $j-1 \not \in I^-$. Therefore, both $I_j$ and $I^-$ are subsets of $[2n-2] \cap (2 \mathbb Z ) \setminus \{j-1\}$, a set of $n - 2$ elements.

Now, as $d_{G_j}(x) \geq \frac{n}{2} + \frac{1}{2}$, we know that $|I_j| \geq \frac{n}{2} - \frac{1}{2}$. Furthermore, as $d^-(v_j) = \frac{n}{2} - \frac{1}{2}$ and $d^+(x)= 0$, we know that $|I^-| = \frac{n}{2} - \frac{1}{2}$. Therefore,
$|I_j| + |I^-| = n -1 ,$ 
and hence there must exist a value $k \in I_j \cap I^-$. Therefore, there exists a Hamiltonian transversal with a vertex sequence
$$(v_{k+1}, v_{k+2}, \dots, v_{j-1}, v_j, v_{k}, v_{k-1}, \dots, v_{j+1}, y,x,v_{k+1}),$$
 as shown in Figure \ref{fig1} (ii). Thus, in this case as well, the proof is complete.
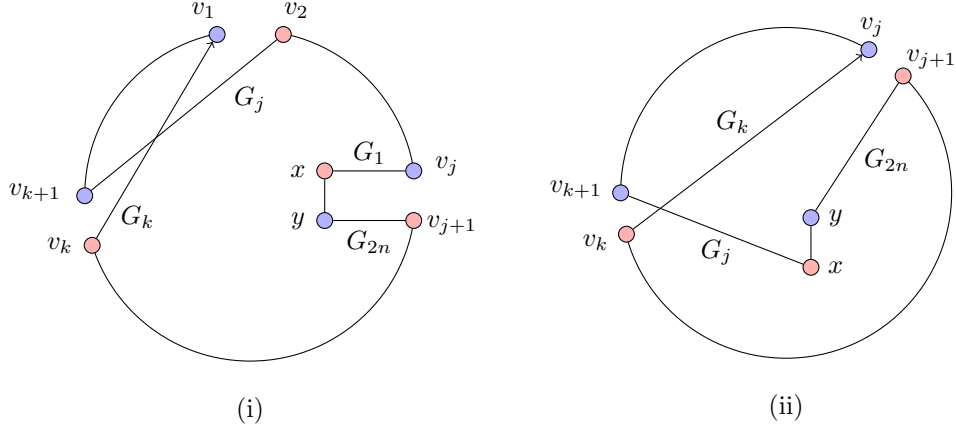
\begin{figure}
\begin{tikzpicture}
[scale=2.2,auto=left,every node/.style={circle,fill=gray!30,minimum size = 6pt,inner sep=0pt}]
\node (z) at (-1.15,-0.3) [fill = white]  {$v_k$};
\node (z) at (1.19,0.18) [fill = white]  {$v_j$};
\node (z) at (-1.3,0.02) [fill = white]  {$v_{k+1}$};
\node (z) at (1.22,-0.18) [fill = white]  {$v_{j+1}$};
\draw (0,0) circle [thick, radius=1];
\node (z) at (0,1) [fill = white] {};
\node (z) at (-0.2,1) [fill = white] {};
\node (z) at (0.2,1) [fill = white] {};
\node (z) at (-0.1,1) [fill = white] {};
\node (z) at (0.1,1) [fill = white] {};
\node (z) at (-0.15,1) [fill = white] {};
\node (z) at (0.15,1) [fill = white] {};
\node (z) at (-0.07,1) [fill = white] {};
\node (z) at (0.07,1) [fill = white] {};

\node (z) at (1,0) [fill = white] {};
\node (z) at (1,-0.1) [fill = white] {};
\node (z) at (1,0.1) [fill = white] {};
\node (z) at (1,-0.15) [fill = white] {};
\node (z) at (1,0.15) [fill = white] {};
\node (z) at (1,-0.07) [fill = white] {};
\node (z) at (1,0.07) [fill = white] {};

\node (z) at (0,-1.3) [fill = white] {(i)};

\node (z) at (-1,0) [fill = white] {};
\node (z) at (-1,-0.1) [fill = white] {};
\node (z) at (-0.97,-0.21) [fill = white] {};
\node (z) at (-1,-0.15) [fill = white] {};
\node (z) at (-1,-0.07) [fill = white] {};
\node (z) at (-1,-0.2) [fill = white] {};
%\node (z) at (-0.33,-0.75) [fill = white]  {$j-1$};
%\node (z) at (0.3,0.75) [fill = white]  {$1$};
%\node (z) at (-0.3,0.75) [fill = white]  {$n$};
%\draw[yellow] [line width=0.5mm] (-0.25,0.97) -- (-0.18,0.6);
%\draw[blue] [line width=0.5mm] (0.25,0.97) -- (0.2,0.6);
%\node (z) at (0,1.15) [fill = white]  {$x_1$};

\node (z) at (0.28,1.12) [fill = white]  {$v_2$};
\node (z) at (-0.28,1.12) [fill = white]  {$v_1$};

\node (v2) at (0.2,0.975) [draw = black, fill = red!30] {};
\node (v1) at (-0.2,0.975) [draw = black, fill = blue!30] {};

%\node (z) at (1.1,0.3) [fill = white]  {$x_i$};

%\node (z) at (-0.36,-0.5) [fill = white]  {$G_k$};
\node (z) at (0,0.58) [fill = white]  {$G_j$};
\node (z) at (-0.68,-0.13) [fill = white]  {$G_k$};

\node (z) at (1.7,0.28) [fill = white]  {};

\node (z) at (0.72,0.26) [fill = white]  {$G_1$};
\node (z) at (0.72,-0.28) [fill = white]  {$G_{2n}$};
%\draw[yellow] [line width=0.5mm] (0.86,0.5) -- (0.5,0.3);
\node (d) at (0.988,0.15) [draw = black, fill = blue!30]  {};
%\draw[blue] [line width=0.5mm] (0.997,0.07) -- (0.6,0.01);
\node (e) at (0.988,-0.15) [draw = black, fill = red!30]  {};

\node (z) at (0.29,0.15) [fill = white]  {$x$};
\node (z) at (0.29,-0.15) [fill = white]  {$y$};
\node (x) at (0.45,0.15) [draw = black, fill = red!30]{};
\node (y) at (0.45,-0.15) [draw = black, fill = blue!30]  {};

%\draw[blue] [line width=0.5mm] (0.752,-0.652) -- (0.45,-0.4);
%\node (z) at (0.752,-0.652) [draw = black]  {};

%\draw[yellow] [line width=0.5mm] (0.95,-0.32) -- (0.58,-0.17);
%\node (z) at (0.95,-0.32) [draw = black]  {};

%\draw[yellow] [line width=0.5mm] (-0.55,-0.83) -- (-0.36,-0.51);
\node (vk) at (-0.955,-0.3) [draw = black, fill = red!30]  {};
%\draw[blue] [line width=0.5mm] (-0.83,-0.55) -- (-0.55,-0.35);
\node (vk1) at (-1,0) [draw = black, fill = blue!30]  {};

%\draw[yellow] [line width=0.5mm] (-0.95,0.3) -- (-0.6,0.17);
%\node (z) at (-0.95,0.3) [draw = black]  {};

%\draw[blue] [line width=0.5mm] (-0.75,0.66) -- (-0.47,0.43);
%\node (z) at (-0.75,0.66) [draw = black]  {};

%\filldraw (0,1) circle (1.8pt);
%\filldraw (0.955,0.3) circle (1.8pt);
%\filldraw (0.866,-0.5) circle (1.8pt);
%\filldraw (-0.7071,-0.7071) circle (1.8pt);
%\filldraw (-0.866,0.5) circle (1.8pt);
 \path[->] (vk) edge  (v1);
 \draw (vk1) edge  (v2);
\foreach \from/\to in {x/d,y/e,x/y}
    \draw (\from) -- (\to);

\end{tikzpicture} 
\begin{tikzpicture}
[scale=2.2,auto=left,every node/.style={circle,fill=gray!30,minimum size = 6pt,inner sep=0pt}]
\node (z) at (-1.15,-0.3) [fill = white]  {$v_{k}$};
\node (z) at (0.52,1) [fill = white]  {$v_j$};
\node (z) at (-1.27,0.02) [fill = white]  {$v_{k+1}$};
\node (z) at (0.88,0.81) [fill = white]  {$v_{j+1}$};
\draw (0,0) circle [thick, radius=1];
\node (z) at (0.5,0.86) [fill = white] {};
\node (z) at (0.55,0.81) [fill = white] {};
\node (z) at (0.6,0.76) [fill = white] {};
\node (z) at (0.67,0.71) [fill = white] {};
\node (z) at (0.63,0.73) [fill = white] {};
\node (z) at (-1,0) [fill = white] {};
\node (z) at (-1,-0.1) [fill = white] {};
\node (z) at (-0.97,-0.21) [fill = white] {};
\node (z) at (-1,-0.15) [fill = white] {};
\node (z) at (-1,-0.07) [fill = white] {};
\node (z) at (-1,-0.2) [fill = white] {};
%\node (z) at (-0.33,-0.75) [fill = white]  {$j-1$};
%\node (z) at (0.3,0.75) [fill = white]  {$1$};
%\node (z) at (-0.3,0.75) [fill = white]  {$n$};
%\draw[yellow] [line width=0.5mm] (-0.25,0.97) -- (-0.18,0.6);
%\draw[blue] [line width=0.5mm] (0.25,0.97) -- (0.2,0.6);
%\node (z) at (0,1.15) [fill = white]  {$x_1$};

%\node (z) at (0.28,1.12) [fill = white]  {$v_2$};
%\node (z) at (-0.28,1.12) [fill = white]  {$v_1$};
\node (z) at (0,-1.3) [fill = white] {(ii)};

%\node (v2) at (0.25,0.97) [draw = black, fill = red!30] {};
%\node (v1) at (-0.25,0.97) [draw = black, fill = blue!30] {};

%\node (z) at (1.1,0.3) [fill = white]  {$x_i$};

%\node (z) at (-0.36,-0.5) [fill = white]  {$G_k$};
\node (z) at (-0.32,0.43) [fill = white]  {$G_k$};

\node (z) at (-0.42,-0.36) [fill = white]  {$G_j$};
\node (z) at (0.61,0.2) [fill = white]  {$G_{2n}$};
%\draw[yellow] [line width=0.5mm] (0.86,0.5) -- (0.5,0.3);
\node (vj) at (0.5,0.866) [draw = black, fill = blue!30]  {};
%\draw[blue] [line width=0.5mm] (0.997,0.07) -- (0.6,0.01);
\node (vj1) at (0.707,0.707) [draw = black, fill = red!30]  {};

\node (z) at (0.3,-0.15) [fill = white]  {$y$};
\node (z) at (0.3,-0.45) [fill = white]  {$x$};
\node (x) at (0.15,-0.15) [draw = black, fill = blue!30]{};
\node (y) at (0.15,-0.45) [draw = black, fill = red!30]  {};

%\draw[blue] [line width=0.5mm] (0.752,-0.652) -- (0.45,-0.4);
%\node (z) at (0.752,-0.652) [draw = black]  {};

%\draw[yellow] [line width=0.5mm] (0.95,-0.32) -- (0.58,-0.17);
%\node (z) at (0.95,-0.32) [draw = black]  {};

%\draw[yellow] [line width=0.5mm] (-0.55,-0.83) -- (-0.36,-0.51);
\node (vk1) at (-0.965,-0.25) [draw = black, fill = red!30]  {};
%\draw[blue] [line width=0.5mm] (-0.83,-0.55) -- (-0.55,-0.35);
\node (vk) at (-1,0) [draw = black, fill = blue!30]  {};

%\draw[yellow] [line width=0.5mm] (-0.95,0.3) -- (-0.6,0.17);
%\node (z) at (-0.95,0.3) [draw = black]  {};

%\draw[blue] [line width=0.5mm] (-0.75,0.66) -- (-0.47,0.43);
%\node (z) at (-0.75,0.66) [draw = black]  {};

%\filldraw (0,1) circle (1.8pt);
%\filldraw (0.955,0.3) circle (1.8pt);
%\filldraw (0.866,-0.5) circle (1.8pt);
%\filldraw (-0.7071,-0.7071) circle (1.8pt);
%\filldraw (-0.866,0.5) circle (1.8pt);
 \path[->] (vk1) edge  (vj);
 %\draw (vk1) edge [ultra thick] (v2);
\foreach \from/\to in {x/vj1,y/vk,x/y}
    \draw (\from) -- (\to);

\end{tikzpicture} 

%\begin{tikzpicture}
%[scale=1.8,auto=left,every node/.style={circle,fill=gray!30}]
%\node (a) at (0.7,0) [fill = white] {$v_j$};
%\node (a) at (0.55,-0.5) [fill = white] {$v_{j+1}$};

%\node (a) at (-1.15,-0.7) [fill = white] {$v_k$};
%\node (a) at (-1.4,0) [fill = white] {$v_{k+1}$};
%\node (v1) at (-1,0) [draw = black] {};
%\node (v6) at (1,0) [draw = black] {};
%\node (a) at (0,0) [fill = white] {$C$};
%\node (a) at (-0.3,1.2) [fill = white] {$v_1$};
%\node (v3) at (-0.3,0.9) [draw = black] {};
%\node (a) at (2.1,-0.5) [fill = white] {$y$};
%\node (a) at (2.1,0) [fill = white] {$x$};
%\node (a) at (1.4,0.2) [fill = white] {$G_1$};
%\node (a) at (1.4,-0.7) [fill = white] {$G_m$};

%\node (a) at (1.4,0) [fill = white] {$s_2$};
%\node (a) at (0.9,-0.9) [fill = white] {$t_2$};
%\node (a) at (0.3,1.2) [fill = white] {$v_2$};
%\node (v4) at (0.3,0.9) [draw = black] {};
%\node (v2) at (-0.9,0.5) [draw = black] {};
%\node (v5) at (0.9,0.5) [draw = black] {};
%\node (v9) at (-0.3,-0.9) [draw = black] {};
%\node (v8) at (0.3,-0.9) [draw = black] {};
%\node (v10) at (-0.9,-0.5) [draw = black] {};

%\node (v7) at (0.9,-0.5) [draw = black] {};
%\node (o1) at (1.75,-0.5) [draw = black] {};
%\node (o2) at (1.75,0) [draw = black] {};
% \draw (v1) edge[ultra thick] (v4);
% \path[->] (v10) edge  (v3);%

%\foreach \from/\to in {v1/v2,v2/v3,v4/v5,v5/v6,v10/v9,v9/v8,v8/v7,v7/v6,v10/v1,o1/o2,o1/v7,o2/v6}
 %   \draw (\from) -- (\to);

%%\draw[transform canvas={xshift=2pt,yshift=-2pt},shorten >= -1pt] (a) -- (b1);
%\end{tikzpicture}
\caption{In both figures, the broken circle represents the edges of the partial transversal $C$ considered in the proof of Theorem \ref{mainThm}.  Figure (i) shows one form of a Hamiltonian transversal found in Case 1 of the proof, and Figure (ii) shows one form of a Hamiltonian transversal found in Case 2 of the proof.}
\label{fig1}
\end{figure}

\section{Bipancyclicity: Proof of Theorem \ref{thmPancyclic}}
This section will be dedicated to proving Theorem \ref{thmPancyclic}. We will continue to use our definitions of $X$ and $\mathcal{G}$ from the previous section. One technique that we will frequently use is that of ``augmenting" a cycle, as in Figure \ref{figAug}. We will establish two lemmas that will be useful when using this technique.

\begin{figure}
\begin{tikzpicture}
[scale=1.5,auto=left,every node/.style={circle,fill=gray!30,minimum size = 6pt,inner sep=0pt}]
\draw (0,0) circle [thick, radius=1];

\node (d) at (0.99,0.15) [draw = black, fill = blue!30]  {};
%\draw[blue] [line width=0.5mm] (0.997,0.07) -- (0.6,0.01);
\node (e) at (0.99,-0.15) [draw = black, fill = red!30]  {};
\node (a) at (0,0) [fill = white] {$C$};
\node (x) at (1.8,0.12) [draw = black, fill = red!30]  {};
\node (y) at (1.8,-0.12) [draw = black, fill = blue!30]  {};

 \draw (x) edge[ultra thick] (d);
 \draw (x) edge[ultra thick] (y);
 \draw (e) edge[ultra thick] (y);

\foreach \from/\to in {x/d,y/e,x/y}
    \draw (\from) -- (\to);
\end{tikzpicture} 
\begin{tikzpicture}
[scale=1.5,auto=left,every node/.style={circle,fill=gray!30,minimum size = 6pt,inner sep=0pt}]
\draw (0,0) circle [thick, radius=1];
\node (d) at (-3,0.5) [fill = white]  {};
\node (d) at (0.97,0.25) [draw = black, fill = blue!30]  {};
\node (f) at (0.97,-0.25) [draw = black, fill = blue!30]  {};
%\draw[blue] [line width=0.5mm] (0.997,0.07) -- (0.6,0.01);
\node (e) at (1,0) [draw = black, fill = red!30]  {};
\node (a) at (0,0) [fill = white] {$C$};
\node (x) at (1.8,0.25) [draw = black, fill = red!30]  {};
\node (y) at (1.8,0) [draw = black, fill = blue!30]  {};
\node (w) at (1.8,-0.25) [draw = black, fill = red!30]  {};

 \draw (x) edge[ultra thick] (d);
 \draw (x) edge[ultra thick] (y);
 \draw (y) edge[ultra thick] (w);
 \draw (w) edge[ultra thick] (f);

\foreach \from/\to in {x/d,x/y,y/w,w/f}
    \draw (\from) -- (\to);
\end{tikzpicture}

\caption{In each figure, the cycle $C$ is ``augmented" by the bolded edges; that is, when the bolded edges are added to the cycle $C$, a cycle longer than $C$ may be found.}
\label{figAug}
\end{figure}
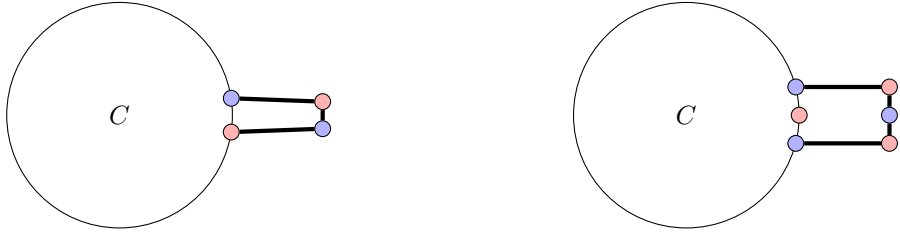

\begin{lemma}
Let $n$ be an integer, and let $(\mathbb{Z}_{2n},+)$ be the cyclic group of $2n$ elements. Let $d \in [1,n]$, and let $A \subseteq \mathbb{Z}_{2n}$. Let $B = (A + d) \cup (A - d)$. If $|A| = |B|$, then $A = A + 2d$.
\label{lemmaNum}
\end{lemma}
\begin{proof}
If $d = n$, then $2d = 0$, and $A = A + 2d$. Otherwise, we define an auxiliary bipartite graph $H$ with partite sets $A$ and $B$. For an element $a \in A$, we let $a$ share an edge with $a+d$ and with $a - d$. By construction, each element $a \in A$ is of degree $2$ in $H$. By a handshaking argument, if $|A| = |B|$, then each element $b \in B$ must also have degree $2$ in $H$. 

We consider an element $a \in A$. The element $a$ has a neighbor $a + d \in B$. As $a + d$ has degree $2$ in $H$, it follows that $a + 2d \in A$; otherwise, no other element of $A$ could be a second neighbor of $a+d$ in $H$. Therefore, $A = A +  2d $.
\end{proof}

\begin{lemma}
Let $n \geq 2$ be an integer, and let $(\mathbb{Z}_{2n},+)$ be the cyclic group of $2n$ elements. Let $k \leq \frac{n}{2}$, and let $A \subseteq \mathbb{Z}_{2n}$ be a set of $k$ odd integers. Let $B = (A + 1) \cup (A - 1)$. Then
\begin{itemize}
\item $|B| \geq |A| + 1$.
\item If $|B| = |A| + 1$, then $A$ is a set of the form $\{a, a+2, a+4, \dots, a+2(k-1)\}$ for some $a \in \mathbb{Z}_{2n}$.
\end{itemize}
\label{lemmaPath}
\end{lemma}
\begin{proof}
For each $a \in A$, clearly $a+1 \in A+1$, so $|B| \geq |A|$. If $|B| = |A|$, then by Lemma \ref{lemmaNum}, $A = A + 2$, from which it follows that $|A| = n$, a contradiction. 

Suppose $|B| = |A| + 1$. Suppose there exists a set $S \subsetneq A$ of size $t$ of the form $\{a, a+2, \dots, a+2(t-1)\}$, and suppose that for each element $x \in (S + 1) \cup (S - 1)$, $\{x+1, x-1\} \cap A \subseteq S$. Then $|(S+ 1) \cup (S - 1)| = t+1$, and letting $A' = A \setminus S$, we have that $|(A' + 1) \cup (A' - 1)| = |A'|$. However, then it follows again from Lemma \ref{lemmaNum} that $|A'| = n$, a contradiction.
\end{proof}

We now consider the graph family $\mathcal G$. Our first aim will be to show that $\mathcal{G}$ contains a partial transversal isomorphic to a cycle of length $2n - 2$, which will be established by the next two claims.

\begin{claim}
$\mathcal{G}$ contains a partial transversal isomorphic to either a cycle of length $2n - 2$ or the disjoint union of a cycle of length $2n - 4$ and a $K_2$.
\label{lemma2n4}
\end{claim}
\begin{proof}
By Theorem \ref{mainThm}, $\mathcal{G}$ contains a partial transversal $C$ isomorphic to a Hamiltonian cycle. Let $C$ have a vertex sequence $(v_1, v_2, \dots, v_{2n}, v_1)$ and an associated bijective function $\phi:E(C) \rightarrow [2n]$. Again, we let $v_2, v_4, \dots, v_{2n}$ be red vertices, and we let $v_1, v_3, \dots, v_{2n-1}$ be blue vertices. 

Without loss of generality, we may assume that for each $v_i \in V(C)$, $\phi(v_i v_{i+1}) = i$, where $v_{2n + 1}$ is identified with $v_1$. We will again define an auxiliary digraph $H$ on $V(C)$. For each red vertex $v_i \in V(C)$, and for each edge $v_i v_j \in E(G_i)$ with $j \neq i+1$, we add the arc $v_i v_j$ to $H$. As each red vertex of $H$ has out-degree at least $\frac{n}{2} - \frac{1}{2}$, it follows that some blue vertex of $H$ has in-degree at least $\frac{n}{2} - \frac{1}{2}$. We may assume without loss of generality that $d^-(v_{1}) \geq \frac{n}{2} - \frac{1}{2}$. We note that if $v_{2n - 2} v_{1}$ is an arc of $H$, then $\mathcal{G}$ contains a partial transversal isomorphic to a cycle of length $2n - 2$ with a vertex sequence $(v_1, v_2, v_3, \dots, v_{2n-3}, v_{2n-2}, v_1)$. Otherwise, we assume that $v_{2n - 2} v_{1}$ is not an arc of $H$. Similarly, we may assume that $v_{2n - 4} v_{1}$ is not an arc of $H$.

Let $k$ be the largest integer for which $v_{k}v_{1}$ is an arc of $H$. We may assume that $k$ is even and $k \leq 2n - 6$. We note that if $v_2 v_{k+3} \in E(G_{1})$ or $v_2 v_{k+5}\in E(G_{1})$, then $\mathcal{G}$ contains a partial transversal isomorphic to one of our desired graphs. Furthermore, for each other arc $v_j v_{1}$ in $H$, if $v_2 v_{j+3} \in E(G_{1})$, then $\mathcal{G}$ contains a partial transversal isomorphic to a cycle of length $2n - 2$. As $d^-(v_{1}) \geq \frac{n}{2} - \frac{1}{2}$, there exist at least $\frac{n}{2} + \frac{1}{2}$ vertices $v_j \in V(C)$ for which the edge $v_2 v_j \in E(G_{1})$ implies a partial transversal isomorphic to one of our desired graphs. As $v_2$ has at least $\frac{n}{2} + \frac{1}{2}$ neighbors via $E(G_{1})$, it then follows that there must exist some edge $v_2 v_j \in E(G_{1})$ that belongs to a partial transversal in $\mathcal{G}$ isomorphic to a cycle of length $2n - 2$ or the disjoint union of a cycle of length $2n - 4$ and a $K_2$.
\end{proof}
\begin{claim}
$\mathcal{G}$ contains a partial transversal isomorphic to a cycle of length $2n - 2$.
\label{lemma2n2}
\end{claim}
\begin{proof}
We may assume from Claim \ref{lemma2n4} that $\mathcal{G}$ contains a partial transversal isomorphic to the disjoint union of a cycle of length $2n - 4$ and a $K_2$.

Let $\mathcal{G}$ contain a partial transversal isomorphic to the disjoint union of a cycle $C$ of length $2n - 4$ and a graph $K \cong K_2$. Let $C$ have a vertex sequence $(v_1, \dots, v_{2n-4}, v_1)$. (Throughout the entire argument, we will identify $v_{2n-3}$ and $v_1$.) Let $V(K) = \{x,y\}$, and let $X \setminus (V(C) \cup V(K)) = \{w,z\}$. Let $ [2n] \setminus \im(\phi) = \{m_1, m_2, m_3\}$. We assume without loss of generality that $x$ and $z$ are red vertices and that $y$ and $w$ are blue vertices. \\ \\
\textbf{Case 1:} $n = 4$ 

In this first case, we seek a partial transversal isomorphic to a cycle of length $6$.
%We let $C$ have a vertex sequence $(v_1, v_2, v_3, v_4, v_1)$. 
We assume without loss of generality that our red vertices are $v_2, v_4, x, z$; hence, our blue vertices are $v_1, v_3, y, w$. For our injective function $\phi:E(C) \cup E(K) \rightarrow [8]$, we assume without loss of generality that $\phi(v_1v_2) = 1$, $\phi(v_2v_3) = 2$, $\phi(v_3v_4) = 3$, $\phi(v_4v_1) = 4$, and $\phi(xy) = 5$. Hence, $[8] \setminus \im(\phi) = \{6,7,8\}$. We seek a partial transversal over $\mathcal G$ isomorphic to a cycle of length $6$.

For this case, we define a \emph{wildcard edge} as an edge $e$ that belongs to $E(G_i)$ for at least $2n - 2 = 6$ values $i \in [8]$. The idea behind this definition is that a wildcard edge may extend any partial transversal with fewer than six edges. We aim to show that if no partial transversal isomorphic to a $6$-cycle exists in $\mathcal G$, then each edge in Figure \ref{fign4} is a wildcard edge.
We observe that $x$ has a neighbor in $\{v_1, v_3\}$ via each of $G_{6}, G_{7}, G_{8}$. If $v_2y \in E(G_{6}) \cup E(G_{7}) \cup E(G_{8})$, then $\mathcal G$ contains a partial transversal isomorphic to a $6$-cycle; hence, we assume that $v_2y \not \in E(G_{6}) \cup E(G_{7}) \cup E(G_{8})$. We may similarly assume that $v_4 y \not \in E(G_{6}) \cup E(G_{7}) \cup E(G_{8})$. It follows that for any edge $e \in E(C)$, we may arbitrarily redefine $\phi(e) = m$ for each $m \in \{6,7,8\}$ while letting $\phi$ still satisfy the conditions of a partial transversal's injective function. Then, by redefining $\phi$ and repeating the same argument at some edge $e \in E(C)$ for each value $m \in \{6,7,8\}$, we may conclude that each edge in $V(C)$ is a wildcard edge, along with the two edges incident to $w$ in Figure \ref{fign4}. Additionally, if there exist two distinct $i,j \in [8]$ for which $v_1 z \in E(G_i)$ and $v_3 z \in E(G_j)$, then $\mathcal G$ contains a partial transversal isomorphic to a $6$-cycle, so we may assume that $yz$ is also a wildcard edge. We may similarly conclude that $xy$ is a wildcard edge. Hence, every edge in Figure \ref{fign4} is a wildcard edge.

Now, we recall that for each $i \in [8]$, $x$ has a neighbor in $\{v_1, v_3\}$ via $G_i$, and we hence note that due to our wildcard edges, if $wz \in E(G_j)$ for any $j \in [8]$, then $\mathcal G$ contains a partial transversal isomorphic to a cycle of length $6$. Therefore, for each value $j \in [8] $, we assume that $zv_1, zv_3 \in E(G_j)$. However, then $\mathcal G$ contains a partial transversal isomorphic to a $6$ cycle using the vertices $x$, $y$, $z$, and three vertices of $V(C)$. Thus, the claim holds in this case. \\ \\
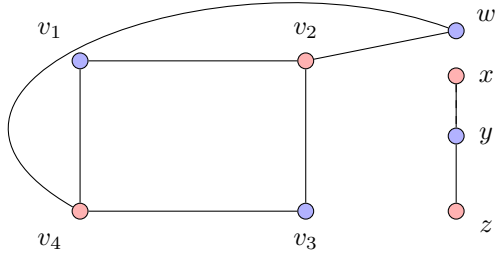
\begin{figure}
\begin{tikzpicture}
[scale=2,auto=left,every node/.style={circle,fill=gray!30,minimum size = 6pt,inner sep=0pt}]
\node (v1) at (-1.5,1) [draw = black, fill = blue!30] {};
\node (v2) at (0,1) [draw = black, fill = red!30] {};
\node (v3) at (0,0) [draw = black, fill = blue!30] {};
\node (v4) at (-1.5,0) [draw = black, fill = red!30] {};

\node (z) at (1,1.2) [draw = black, fill = blue!30] {};
\node (x) at (1,0.9) [draw = black, fill = red!30] {};
\node (y) at (1,0.5) [draw = black, fill = blue!30] {};
\node (w) at (1,0) [draw = black, fill = red!30] {};

%\node (zzzz) at (2.2,1.2) [fill = white]  {};
\node (zzzz) at (-1.7,1.2) [fill = white]  {$v_{1}$};
\node (zzzz) at (0,1.2) [fill = white]  {$v_{2}$};
\node (zzzz) at (0,-0.2) [fill = white]  {$v_{3}$};
\node (zzzz) at (-1.7,-0.2) [fill = white]  {$v_{4}$};
\node (zzzz) at (1.2,1.3 ) [fill = white]  {$w$};
\node (zzzz) at (1.2,0.9 ) [fill = white]  {$x$};
\node (zzzz) at (1.2,0.5 ) [fill = white]  {$y$};
\node (zzzz) at (1.2,-0.1 ) [fill = white]  {$z$};
%\node (zzzz) at (-0.50,1.6) [fill = white]  {$678$};
%\node (zzzz) at (-1.85,0.5) [fill = white]  {$4$};
%\node (zzzz) at (-0.15,0.5) [fill = white]  {$2$};
%\node (zzzz) at (-1,-0.15) [fill = white]  {$3$};
%\node (zzzz) at (-1,1.15) [fill = white]  {$1$};
%\node (zzzz) at (0.9,0.7) [fill = white]  {$5$};

\draw (z) to [out=160,in=150,looseness = 01.50] (v4);
\draw[dashed] (x) to (y);

\foreach \from/\to in {v1/v2,v2/v3,v3/v4,v4/v1,v2/z,y/w,x/y}
    \draw (\from) -- (\to);

%\draw[transform canvas={xshift=2pt,yshift=-2pt},shorten >= -1pt] (a) -- (b1);
\end{tikzpicture}
\caption{The figure shows the vertices of $X$ as considered in Lemma \ref{lemma2n2} in the case $n = 4$. The aim in the proof of this case of Lemma \ref{lemma2n2} is to show that each edge shown in the figure is a \emph{wildcard edge}---that is, an edge that belongs to at least $2n - 2 = 6$ graphs of $\mathcal G$. As these wildcard edges may extend any partial transversal of length less than $6$, they are useful for finding partial transversals over $\mathcal G$.}
\label{fign4}
\end{figure}
%Reca $C$ have a vertex sequence $(v_1, v_2, \dots, v_6, v_1)$. 
  \\ 
\textbf{Case 2:} $n \geq 5$ 

In this second case, we seek a partial transversal isomorphic to a cycle of length $2n - 2$. %We consider two subcases.
%\begin{enumerate}
%\item First, suppose that all neighbors of $x$ and $y$ via the graphs $G_{m_1}$ and $G_{m_2}$ belong to $V(C) \cup \{x,y\}$. We let 
%$$A:=\{v \in V(C): vx \in E(G_{m_1})\},$$
%and we let $N_C(A) \subseteq V(C)$ denote the set of vertices adjacent to $A$ in $C$. Recall that $x$ is a red vertex, $A$ contains blue vertices, and $N_C(A)$ contains red vertices. As $|A| \geq \frac{n}{2} - \frac{1}{2}$, we see that $|N_C(A)| \geq \frac{n}{2} + \frac{1}{2}$, and hence at most $\frac{n}{2} - \frac{5}{2}$ red vertices belong to $V(C) \setminus N_C(A)$. As $y$ has at least $\frac{n}{2} - 1$ neighbors in $V(C)$ via the graph $G_{m_2}$, there must exist vertices $v_j, v_{j+1} \in V(C)$ such that $x$ has a neighbor in $\{v_j, v_{j+1}\}$ via the graph $G_{m_1}$, and $y$ has a neighbor in $\{v_j, v_{j+1}\}$ via the graph $G_{m_2}$. Then we may find a partial transversal over $\mathcal{G}$ isomorphic to a cycle of length $2n - 2$ containing the vertices $V(C) \cup \{x,y\}$.
%\item Now suppose, on the other hand, that one of $x,y$ has a neighbor $z$ or $w$ in a graph $G_{m_1}$ or $G_{m_2}$. Without loss of generality, we assume that $yz \in E(G_{m_1})$. 
We let 
$$A:=\{v \in V(C): vx \in E(G_{m_1})\},$$
and we let $N_C(A) \subseteq V(C)$ denote the set of vertices adjacent to $A$ in $C$. We observe that $|A| \geq \frac{n}{2} - \frac{3}{2}$. If $|N_C(A)| \geq \frac{n}{2} + \frac{1}{2}$, then at most $\frac{n}{2} - \frac{5}{2}$ red vertices belong to $V(C) \setminus N_C(A)$. As $y$ has at least $\frac{n}{2} - 2$ neighbors in $V(C)$ via the graph $G_{m_2}$, there must exist vertices $v_j, v_{j+1} \in V(C)$ such that $x$ has a neighbor in $\{v_j, v_{j+1}\}$ via the graph $G_{m_1}$ and $y$ has a neighbor in $\{v_j, v_{j+1}\}$ via the graph $G_{m_2}$. Then we may find a partial transversal over $\mathcal{G}$ isomorphic to a cycle of length $2n - 2$ containing the vertices $V(C) \cup \{x,y\}$.

Otherwise, suppose $|N_C(A)| \leq \frac{n}{2}$.
If $y$ has at least $\frac{n}{2} - \frac{3}{2}$ neighbors in $V(C)$ via $G_{m_2}$, then by a similar argument, we may find a partial transversal over $\mathcal{G}$ isomorphic to a cycle of length $2n - 2$ containing the vertices $V(C) \cup \{x,y\}$. Hence, we assume instead that $y$ has exactly $\frac{n}{2} - 2$ neighbors in $V(C)$ via $G_{m_2}$ and hence that $yz \in E(G_{m_2})$.
Then we apply Lemma \ref{lemmaPath}, which tells us that $A$ is of the form 
$$A = \{v_i, v_{i+2}, v_{i+4}, \dots, v_{i - 2+ 2\lceil \frac{n-3}{2}\rceil}\},$$ for some $v_i \in V(C)$. Then, as $n \geq 5$, it follows that there exist at least $|A| + 2 \geq \frac{n}{2} + \frac{1}{2}$ vertices in $C$ at a distance of exactly $2$ from a vertex of $A$. Hence, as $z$ has at least $\frac{n}{2} - \frac{3}{2}$ neighbors in $V(C)$ via the graph $G_{m_2}$, $z$ must have some neighbor in $C$ via $G_{m_2}$ that is at distance exactly $2$ from a vertex of $A$. Therefore, there exist two vertices $v_j, v_{j+2} \in V(C)$ such that $x$ has a neighbor in $\{v_j, v_{j+2}\}$ via the graph $G_{m_1}$ and $z$ has a distinct neighbor in $\{v_j, v_{j+2}\}$ via $G_{m_2}$. Therefore, $\mathcal{G}$ contains a partial transversal isomorphic to a cycle of length $2n - 2$ on the vertex set $V(C) \setminus \{v_{j+1}\} \cup \{x,y,z\}$.
%\end{enumerate}
\end{proof}

We have finally established that $\mathcal{G}$ contains a partial transversal isomorphic to a cycle of length $2n - 2$. Now we suppose, for the sake of contradiction, that $\mathcal{G}$ does not contain a partial transversal isomorphic to a cycle of length $\ell$, for some even integer $4 \leq \ell \leq n-4$. Let $C$ be a partial transversal over $\mathcal{G}$ isomorphic to a cycle of length $2n - 2$ with an associated injective function $\phi:E(C) \rightarrow [2n]$. We let $C$ have a vertex sequence $(v_1, \dots, v_{2n-2}, v_1)$. We let $X \setminus V(C) = \{x,y\}$. For each $i \in [2n-2]$, we assume without loss of generality that $\phi(v_iv_{i+1}) = i$, with $v_{2n - 1}$ and $v_1$ identified. Then it follows that $[2n] \setminus \im(\phi) = \{2n-1, 2n\}$. We again let our red vertices be $x,v_2, v_4, \dots, v_{2n-2}$, and we let our blue vertices be $y,  v_1, v_3, \dots, v_{2n-3}$.

For the sake of convenience, we will endow $V(C)$ with some group operations. For a vertex $v_i \in V(C)$ and an integer $k$, we define $v_i + k = v_{i+k}$, again with $v_{2n - 1}$ ``overflowing" to $v_1$. By this definition, $V(C)$ is isomorphic as a group to $(\mathbb{Z}_{2n-2},+)$. For a set $A \subseteq V(C)$ and an integer $k$, we define $A+k = \{v+k:v \in A\}$. We make several observations.

\begin{claim}
\label{claimAnyColor}
For each $v_j \in V(C)$, $v_j v_{j+1} \in E(G_{2n-1})$ or $v_j v_{j+1} \in E(G_{2n})$.
\end{claim}
\begin{proof}
To prove this claim, we will apply an idea of Bondy \cite{Bondy} which is commonly used in proving pancyclicity results (c.f. \cite{Cheng}, \cite{Schmeichel}). Suppose that $v_j v_{j+1} \not \in E(G_{2n-1})$ and $v_j v_{j+1} \not \in E(G_{2n})$. For edges in $C$ incident to $v_j$ and $v_{j+1}$, we will create the following pairings for each $k \in [2n - 2] \setminus \{j+1\}$ of parity opposite $j$:
\begin{itemize}
\item If $j \not \in \{k-\ell + 3, k - \ell+4, \dots, k-2, k-1\}$, then we pair $v_jv_k$ with $v_{j+1} v_{k-\ell+3}$. 
\item If $j \in \{k-\ell + 3, k - \ell+4, \dots, k-2, k-1\}$, then we pair $v_j v_k$ with $v_{j+1} v_{k-\ell+1}$. 
%\item For $j + \ell - 1 \leq k \leq j - 1$, pair $v_jv_k$ with $v_{j+1} v_{k-\ell+3}$. 
%\item For $j + 3 \leq k \leq j + \ell - 3$, pair $v_j v_k$ with $v_{j+1} v_{k-\ell+1}$. 
\end{itemize}
It is clear from Figure \ref{figBondy} that if $\mathcal{G}$ does not contain a partial transversal isomorphic to a cycle of length $\ell$, then for any pair $(a,b)$ given above, if $a \in E(G_{2n-1})$, then $b \not \in E(G_{2n})$. One may also check that each edge incident to $v_j$ is paired with exactly one edge incident to $v_{j+1}$, and one may also check that no edge $v_j v_k$ is paired with $v_{j+1} v_j$.
%\item If $a \in E(G_{m_2})$, then $b \not \in E(G_{m_1})$.
%\item If $b \in E(G_{m_1})$, then $a \not \in E(G_{m_2})$.
%\item If $b \in E(G_{m_2})$, then $a \not \in E(G_{m_1})$.
%\end{itemize}
Therefore, for the pair $v_j, v_{j+1}$, 
$$|N_{G_{2n-1}}(v_j) \cap V(C)| + |N_{G_{2n}}(v_{j+1}) \cap V(C)| \leq n-2.$$
 Then, as each of $v_j, v_{j+1}$ has at most one neighbor outside of $V(C)$ via any graph, it follows that $d_{G_{2n-1}}(v_j) + d_{G_{2n}}(v_{j+1})  \leq n$, a contradiction.
\end{proof}

\begin{figure}
\begin{tikzpicture}
[scale=2.2,auto=left,every node/.style={circle,fill=gray!30,minimum size = 6pt,inner sep=0pt}]
\node (z) at (-0.87,0.83) [fill = white]  {$v_{k-\ell+3}$};
\node (z) at (0,1.15) [fill = white]  {$\ell-3$};
\node (z) at (0.87,0.76) [fill = white]  {$v_k$};
%\node (z) at (0.66,0.58) [fill = white]  {$m_1$};
\node (z) at (0.28,-1.15) [fill = white]  {$v_j$};
\node (z) at (-0.18,-1.15) [fill = white]  {$v_{j+1}$};
\draw (0,0) circle [thick, radius=1];
\draw [very thick] (0.8,0.6) arc[start angle=36.9,delta angle=106.2,radius=1];
\draw [very thick] (-0.18,-0.98) arc[start angle=260,delta angle=25,radius=1];
%\node (z) at (-0.15,1) [fill = white] {};

%\node (z) at (-0.33,-0.75) [fill = white]  {$j-1$};
%\node (z) at (0.3,0.75) [fill = white]  {$1$};
%\node (z) at (-0.3,0.75) [fill = white]  {$n$};
%\draw[yellow] [line width=0.5mm] (-0.25,0.97) -- (-0.18,0.6);
%\draw[blue] [line width=0.5mm] (0.25,0.97) -- (0.2,0.6);
%\node (z) at (0,1.15) [fill = white]  {$x_1$};

%\node (z) at (0.28,1.12) [fill = white]  {$q_1$};
%\node (z) at (-0.18,1.12) [fill = white]  {$p_1$};

\node (p2) at (0.8,0.6) [draw = black, fill = blue!30] {};
\node (qn) at (-0.8,0.6) [draw = black, fill = red!30] {};
%\node (pn) at (-0.86,0.5) [draw = black, fill = blue!30]  {};

%\node (q1) at (0.25,0.97) [draw = black, fill = red!30] {};
%\node (p1) at (-0.18,0.98) [draw = black, fill = blue!30] {};
\node (qj1) at (0.25,-0.97) [draw = black, fill = red!30] {};
\node (pj) at (-0.18,-0.98) [draw = black, fill = blue!30] {};

%\node (z) at (1.1,0.3) [fill = white]  {$x_i$};

%\node (z) at (0.58,-0.2) [fill = white]  {$m_2$};
%\node (z) at (-0.15,0.17) [fill = white]  {$m_1$};

%\draw[yellow] [line width=0.5mm] (0.86,0.5) -- (0.5,0.3);
%\node (q2) at (0.86,0.5) [draw = black, fill = red!30]  {};
%\draw[blue] [line width=0.5mm] (0.997,0.07) -- (0.6,0.01);
%\node (p3) at (0.997,0.07) [draw = black, fill = blue!30]  {};

%\draw[blue] [line width=0.5mm] (0.752,-0.652) -- (0.45,-0.4);
%\node (z) at (0.752,-0.652) [draw = black]  {};

%\draw[yellow] [line width=0.5mm] (0.95,-0.32) -- (0.58,-0.17);
%\node (z) at (0.95,-0.32) [draw = black]  {};

%\draw[yellow] [line width=0.5mm] (-0.95,0.3) -- (-0.6,0.17);
%\node (z) at (-0.95,0.3) [draw = black]  {};

%\draw[blue] [line width=0.5mm] (-0.75,0.66) -- (-0.47,0.43);
%\node (z) at (-0.75,0.66) [draw = black]  {};

%\filldraw (0,1) circle (1.8pt);
%\filldraw (0.955,0.3) circle (1.8pt);
%\filldraw (0.866,-0.5) circle (1.8pt);
%\filldraw (-0.7071,-0.7071) circle (1.8pt);
%\filldraw (-0.866,0.5) circle (1.8pt);
 \draw[dashed] (qj1) edge  (p2);
 \draw[dashed] (pj) edge  (qn);

\end{tikzpicture} 
\begin{tikzpicture}
[scale=2.2,auto=left,every node/.style={circle,fill=gray!30,minimum size = 6pt,inner sep=0pt}]
\node (z) at (-1.25,-0.8) [fill = white]  {$k-j-1$};
\node (z) at (1.1,-1.05) [fill = white]  {$j-k+\ell - 1$};
\node (z) at (1.2,-0.53) [fill = white]  {$v_{k - \ell + 1}$};
\node (z) at (-1.15,0) [fill = white]  {$v_k$};
%\node (z) at (0.67,0.96) [fill = white]  {$p_2$};
%\node (z) at (0.66,0.58) [fill = white]  {$m_1$};
\node (z) at (0.28,-1.15) [fill = white]  {$v_j$};
\node (z) at (-2,-1.15) [fill = white]  {};
\node (z) at (-0.18,-1.15) [fill = white]  {$v_{j+1}$};
\draw (0,0) circle [ radius=1];
\draw [very thick] (-1,0)  arc[start angle=180,delta angle=80,radius=1];
\draw [very thick] (0.866,-0.5)  arc[start angle=-30,delta angle=-45,radius=1];
%\node (z) at (-0.15,1) [fill = white] {};
%\node (z) at (-0.33,-0.75) [fill = white]  {$j-1$};
%\node (z) at (0.3,0.75) [fill = white]  {$1$};
%\node (z) at (-0.3,0.75) [fill = white]  {$n$};
%\draw[yellow] [line width=0.5mm] (-0.25,0.97) -- (-0.18,0.6);
%\draw[blue] [line width=0.5mm] (0.25,0.97) -- (0.2,0.6);
%\node (z) at (0,1.15) [fill = white]  {$x_1$};

%\node (z) at (0.28,1.12) [fill = white]  {$q_1$};
%\node (z) at (-0.18,1.12) [fill = white]  {$p_1$};

%\node (p2) at (0.6,0.8) [draw = black, fill = blue!30] {};
%\node (qn) at (-0.6,0.8) [draw = black, fill = red!30] {};
\node (pn) at (-1,0) [draw = black, fill = blue!30]  {};

%\node (q1) at (0.25,0.97) [draw = black, fill = red!30] {};
%\node (p1) at (-0.18,0.98) [draw = black, fill = blue!30] {};
\node (qj1) at (0.25,-0.97) [draw = black, fill = red!30] {};
\node (pj) at (-0.18,-0.98) [draw = black, fill = blue!30] {};

%\node (z) at (1.1,0.3) [fill = white]  {$x_i$};

%\node (z) at (0.58,-0.2) [fill = white]  {$m_2$};
%\node (z) at (-0.15,0.17) [fill = white]  {$m_1$};

%\draw[yellow] [line width=0.5mm] (0.86,0.5) -- (0.5,0.3);
\node (q2) at (0.866,-0.5) [draw = black, fill = red!30]  {};
%\draw[blue] [line width=0.5mm] (0.997,0.07) -- (0.6,0.01);
%\node (p3) at (0.997,0.07) [draw = black, fill = blue!30]  {};

%\draw[blue] [line width=0.5mm] (0.752,-0.652) -- (0.45,-0.4);
%\node (z) at (0.752,-0.652) [draw = black]  {};

%\draw[yellow] [line width=0.5mm] (0.95,-0.32) -- (0.58,-0.17);
%\node (z) at (0.95,-0.32) [draw = black]  {};

%\draw[yellow] [line width=0.5mm] (-0.95,0.3) -- (-0.6,0.17);
%\node (z) at (-0.95,0.3) [draw = black]  {};

%\draw[blue] [line width=0.5mm] (-0.75,0.66) -- (-0.47,0.43);
%\node (z) at (-0.75,0.66) [draw = black]  {};

%\filldraw (0,1) circle (1.8pt);
%\filldraw (0.955,0.3) circle (1.8pt);
%\filldraw (0.866,-0.5) circle (1.8pt);
%\filldraw (-0.7071,-0.7071) circle (1.8pt);
%\filldraw (-0.866,0.5) circle (1.8pt);
 \draw[dashed] (qj1) edge  (pn);
 \draw[dashed] (pj) edge  (q2);

\end{tikzpicture} 
\caption{In both figures, the circle represents a partial transversal $C$ isomorphic to a cycle of length $2n - 2$ with an injective function $\phi$ satisfying $[2n] \setminus \im(\phi) = \{2n-1, 2n\}$. In both figures, if one dashed edge belongs to $E(G_{2n-1})$ and the other belongs to $E(G_{2n})$, then $\mathcal G$ contains a partial transversal isomorphic to a cycle of length $\ell$ containing the bolded edges and the dashed edges. The number of edges represented by each bolded arc, when more than one, is indicated.}
\label{figBondy}
\end{figure}
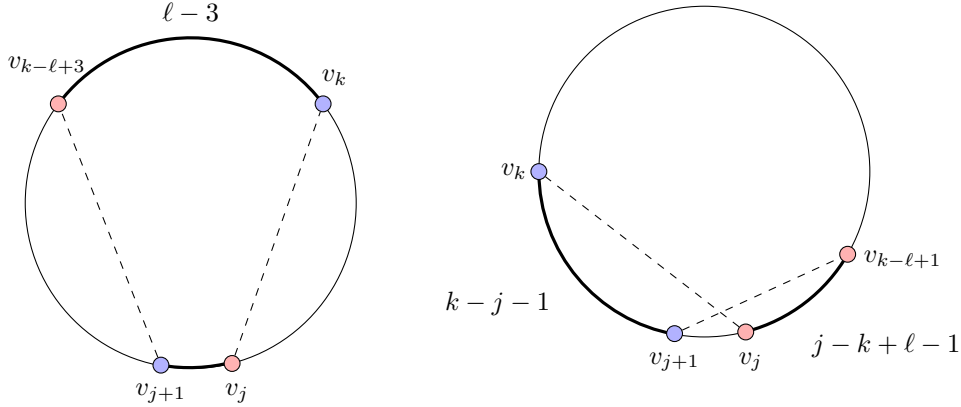

Using Claim \ref{claimAnyColor}, for any edge $v_i v_{i+1} \in E(C)$ for which $\phi(v_i v_{i+1}) = i$, we may redefine $\phi$ so that $\phi(v_i v_{i+1}) \in \{2n-1, 2n\}$. As a result, we may define our injective function $\phi$ so that $\phi$ ``misses" any given value $i \in [2n]$. This observation will help us prove our next claim, which shows that $xy$ is a \emph{wildcard edge}, where an edge $e$ is said this time to be a wildcard edge if $e \in E(G_i)$ for all $i \in [2n]$.

\begin{claim}
For every graph $G_i \in \mathcal{G}$, $xy \in E(G_i)$.
\label{claimAllColors}
\end{claim}
\begin{proof}
Recall that $x$ is a red vertex. First, we note that by Claim \ref{claimAnyColor}, given any graph $G_i \in \mathcal{G}$, we may find a partial transversal over $\mathcal{G}$ on $E(C)$ with an associated injective function $\psi$ such that $i \in [2n] \setminus \im(\psi)$. Therefore, it suffices to consider our original partial transversal $C$ and associated function $\phi$ and to show that $xy \in E(G_{2n-1})$ and $ xy \in E(G_{2n})$.

Suppose for the sake of contradiction that $xy \not \in E(G_{2n})$. Then all neighbors of $x$ via $G_{2n}$ belong to $V(C)$. Let $A$ denote the set of neighbors of $x$ via the graph $G_{2n}$. In particular, the set $B = A + (\ell-2)$ has at least $\frac{n}{2} + \frac{1}{2}$ elements. As $x$ has at least $\frac{n}{2} - \frac{1}{2}$ neighbors in $V(C)$ via the graph $G_{2n-1}$, and as at most $(n-1) - (\frac{n}{2} + \frac{1}{2}) =\frac{n}{2} - \frac{3}{2} $ blue vertices of $V(C)$ do not belong to $B$, it follows that some neighbor of $x$ via $G_{2n-1}$ belongs to $B$. Then it follows that $\mathcal{G}$ in fact contains a partial transversal isomorphic to a cycle of length $\ell$. Therefore, we may assume that $xy \in E(G_{2n})$, and by a similar argument, we may assume that $xy \in E(G_{2n-1})$.
\end{proof}

Claim \ref{claimAllColors} shows us that $xy$ is a wildcard edge in $\mathcal{G}$, which gives us the last piece that we need to finish our proof. Let $A \subseteq V(C)$ be the set of vertices in $C$ adjacent to $x$ via $G_{2n-1}$. Clearly $|A| \geq \frac{n}{2} - \frac{1}{2}$. Let $B = (A + (\ell-2)) \cup (A - (\ell-2))$. If $|B| \geq \frac{n}{2} + \frac{1}{2}$, then by the argument in Claim \ref{claimAllColors}, $x$ has a neighbor in $B$ via the graph $G_{2n}$, and then $\mathcal{G}$ contains a partial transversal isomorphic to a cycle of length $\ell$. Otherwise, $|A| = |B| \leq \frac{n}{2}$, and then by Lemma \ref{lemmaNum}, $A = A + ( 2\ell - 4 )$.

Next, let $B' = (A + (\ell - 3)) \cup (A - (\ell - 3))$. As $xy$ is an edge of all graphs in $\mathcal{G}$, if $y$ has a neighbor in $B'$ by $G_{2n}$, then $\mathcal{G}$ has a partial transversal isomorphic to a cycle of length $\ell$. If $|B'| \geq \frac{n}{2} + \frac{1}{2}$, then $y$ has a neighbor in $B'$ via $G_{2n}$, and we find our partial transversal in $\mathcal G$ of length $\ell$. Otherwise, $|A| = |B'| \leq \frac{n}{2}$, and by Lemma \ref{lemmaNum}, $A = A + ( 2l - 6)$. 

As $A = A + ( 2 \ell - 4 )$, and as $A = A + ( 2\ell - 6 )$, it follows that $A = A + 2 $. Then it follows that $|A| = |B| = n - 1$, and $x$ has a neighbor in $B$ via $G_{2n}$. Hence, in all cases, $\mathcal{G}$ contains a partial transversal isomorphic to a cycle of length $\ell$. This completes the proof. \qed

\section{Perfect matchings: Proof of Theorem \ref{thmPM}}
This section will be dedicated to the proof of Theorem \ref{thmPM}. In this section, we define $X$ as a set of $n$ blue vertices $\{p_1, \dots, p_n\}$ and $n$ red vertices $\{q_1, \dots, q_n\}$. This time, we assume $\mathcal{G} = \{G_1, \dots, G_n\}$ is a set of $n$ bipartite graphs on the vertex set $X$, such that for each $G_i \in \mathcal{G}$, the red vertices of $X$ make up one color class of $G_i$, and the blue vertices of $X$ make up the other color class of $G_i$. We will show that under the conditions of Theorem \ref{thmPM}, $\mathcal{G}$ contains a transversal isomorphic to a perfect matching.

Supposing that Theorem \ref{thmPM} does not hold, we consider an edge-maximal counterexample $\mathcal G$ to the theorem. That is, we let $\mathcal{G}$ be a family of graphs such that after adding any edge to any graph $G_i \in \mathcal{G}$, the resulting family contains a perfect matching transversal. If each graph $G_i \in \mathcal{G}$ is a complete bipartite graph, then $\mathcal{G}$ certainly contains a perfect matching transversal. Therefore, we may assume without loss of generality that $G_n$ is not a complete bipartite graph and that we may add some edge $e$ to $G_n$ with a red endpoint and a blue endpoint. As $\mathcal G$ is an edge-maximal counterexample, $\mathcal{G} + e = \{G_1, G_2, \dots, G_{n}+e\}$ contains a perfect matching transversal $M$ with an associated bijective function $\phi:E(M) \rightarrow [n]$. We may assume without loss of generality that $e = p_nq_n$, and that $M$ consists of edges $p_1q_1, p_2q_2, \dots, p_nq_n$, and that for each edge $p_iq_i$, $\phi(p_iq_i) = i$. It follows that $\mathcal{G}$ contains a partial transversal $p_1q_1, \dots, p_{n-1}q_{n-1}$ for which each edge $p_iq_i$  ($1 \leq i \leq n-1$) belongs to $E(G_i)$.

The remainder of the proof will rely heavily on ideas of Joos and Kim from \cite{Joos}. We define an auxiliary digraph $H$ for which $V(H) = X$. For each value $i$, $1 \leq i \leq n-1$, and for each edge $e \in E(G_i)$ incident to $q_i$, $e \neq p_i q_i$, we add an arc $e$ to $H$ directed away from $q_i$. We claim that in $H$, $d^-(p_n) \leq \frac{n}{2} - \frac{3}{2}$. Indeed, suppose that $d^-(p_n) \geq \frac{n}{2} - 1$. We have assumed that $p_nq_n \not \in E(G_n)$, and hence we may assume that $q_n$ has at least $\frac{n}{2} + \frac{1}{2}$ neighbors among $p_1, \dots, p_{n-1}$ via $G_n$. As $p_n$ has at least $\frac{n}{2} - 1$ in-neighbors in $H$ among $q_1, \dots, q_{n-1}$, and as $(\frac{n}{2} + \frac{1}{2}) + (\frac{n}{2} - 1) > n - 1$, it follows from the pigeonhole principle that there exists a pair $p_i, q_i$ such that $p_i q_n$ is an arc of $H$ and $p_i q_n \in E(G_n)$. This implies that a perfect matching transversal exists over $\mathcal G$ containing the edges $p_i q_n$ and $p_n q_i$. Therefore, we assume that $d^-(p_n) \leq \frac{n}{2} - \frac{3}{2}$.

We first consider the case that $n$ is even. In this case, the total number of arcs in $H$ not ending at $p_n$ is at least $(n-1)(\frac{n}{2}) - (\frac{n}{2} - 2)> (n-1)(\frac{n}{2} - 1)$. Therefore, there exists a vertex $p_i$, $1 \leq i \leq n-1$, for which $d^-(p_i) \geq \frac{n}{2}$. As $q_n$ has at least $\frac{n}{2}+1$ neighbors among $p_1, \dots, p_{n-1}$ via $G_n$, and as $p_n$ has at least $\frac{n}{2}-1$ neighbors among $q_1, \dots, q_{n-1}$ via $G_i$, there exists a pair $p_j, q_j$ for which $q_j p_n \in E(G_i)$ and $p_j q_n \in E(G_n)$. We may assume that $i \neq j$, as otherwise there exists a perfect matching transversal over $\mathcal G$ containing $p_j q_n$ and $q_j p_n$. Finally, as we may assume that $p_i q_i \not \in E(G_j)$, we see that $q_i$ has at least $\frac{n}{2}$ neighbors among $p_1, \dots, p_{i-1}, p_{i+1}, \dots, p_{n-1}$ via $G_j$, and $p_i$ has at least $\frac{n}{2}$ in-neighbors in $H$ among $q_1, \dots, q_{i-1}, q_{i+1}, \dots, q_{n-1}$. Therefore, we may choose a pair $p_k, q_k$, $k \neq j$ for which $p_i q_k$ is an arc of $H$ and $p_k q_i \in E(G_j)$. Then it follows that there exists a perfect matching transversal over $\mathcal G$ containing the edges $p_k q_i$ and $q_k p_i$, and the proof is complete.

Next, we consider the case that $n$ is odd. In this case, the total number of arcs in $H$ not ending at $p_n$ is at least $(n-1)(\frac{n}{2}-\frac{1}{2}) - (\frac{n}{2} - \frac{3}{2}) = \frac{1}{2}n^2 - \frac{3}{2}n + 2$. We further consider two cases. \\ \\
\textbf{Case 1:} There exists a vertex $q_i$, $1 \leq i \leq n-1$, with in-degree at least $\frac{n}{2} + \frac{1}{2}$. 

As $q_n$ has at least $\frac{n}{2} + \frac{1}{2}$ neighbors among $p_1, \dots, p_{n-1}$ via $G_n$, and as $p_n$ has at least $\frac{n}{2} - \frac{1}{2}$ neighbors among $q_1, \dots, q_{n-1}$ via $G_i$, there exists a pair $p_j, q_j$ for which $p_j q_n \in E(G_n)$ and $q_j p_n \in E(G_i)$. We may assume that $i \neq j$, as otherwise there exists a perfect matching transversal over $\mathcal G$ containing $p_j q_n$ and $q_j p_n$. Finally, as we may assume that $p_i q_i \not \in E(G_j)$, we see that $q_i$ has at least $\frac{n}{2} - \frac{1}{2}$ neighbors among $p_1, \dots, p_{i-1}, p_{i+1}, \dots, p_{n-1}$ via $G_j$, and $p_i$ has at least $\frac{n}{2} + \frac{1}{2}$ in-neighbors in $H$ among $q_1, \dots, q_{i-1}, q_{i+1}, \dots, q_{n-1}$. Therefore, we may choose a pair $p_k, q_k$, $k \neq j$ for which $ p_i q_k $ is an arc of $H$ and $p_k q_i  \in E(G_j)$. Then it follows that there exists a perfect matching transversal over $\mathcal G$ containing the edges $p_k q_i$ and $q_k p_i$, and the proof is complete. \\ \\
\textbf{Case 2:} There exists no vertex $p_i$, $1 \leq i \leq n-1$, with in-degree at least $\frac{n}{2} + \frac{1}{2}$. 

In this case, we let $a$ denote the number of vertices among $p_1, \dots, p_{n-1}$ with in-degree $\frac{n}{2} - \frac{1}{2}$. By counting the arcs of $H$ not ending at $p_n$, we see that 
$$a(\frac{n}{2} - \frac{1}{2}) + (n-1-a)(\frac{n}{2} - \frac{3}{2}) \geq \frac{1}{2}n^2 - \frac{3}{2}n + 2,$$ from which it follows that $a \geq \frac{n}{2} + \frac{1}{2}$. As $p_n$ has at least $\frac{n}{2} + \frac{1}{2}$ neighbors among $q_1, \dots, q_{n-1}$, there must exist a pair $p_j, q_j$ for which $ p_n q_j \in E(G_n)$ and $d^-(p_j) = \frac{n}{2} - \frac{1}{2}$. Finally, we may assume that $p_j q_n \not \in E(G_j)$, as otherwise, there would exist a perfect matching transversal over $\mathcal G$ containing $p_j q_n$ and $q_j p_n$. Therefore, as $p_j$ has $\frac{n}{2} - \frac{1}{2}$ in-neighbors among $q_1, \dots, q_{j-1}, q_{j+1}, \dots,  q_{n-1}$, and as $q_n$ has at least $\frac{n}{2} - \frac{1}{2}$ neighbors among $p_1, \dots, p_{j-1}, p_{j+1}, \dots, p_{n-1}$ via $G_j$, there must exist a pair $p_k, q_k$ for which $q_k p_j$ is an arc of $H$ and $p_k q_n \in E(G_j)$. Then it follows that there exists a perfect matching transversal over $\mathcal G$ containing the edges $p_j q_k$ and $p_k q_n$, and the proof is complete. \qed \\ 

We give a construction to show that the degree condition of Theorem \ref{thmPM} is in some sense best possible. First, we consider the case that $n$ is even, and we show that a minimum degree of $\frac{n}{2}$ for all vertices does not imply the existence of a perfect matching transversal. We partition our red vertices of $X$ into two equally sized parts $A,B$, and we partition our blue vertices of $X$ into two equally sized parts $C,D$. For $1 \leq i \leq n-1$, we let $G_i$ be the union of the complete graphs on the pairs $(A,C)$ and $(B,D)$. We let $G_n$ be the union of the complete graphs on the pairs $(A,D)$ and $(B,C)$. We let $\mathcal{G} = \{G_1, \dots, G_n\}$, and we see that for each $G_i \in \mathcal{G}$, $\delta(G_i) = \frac{n}{2}$. 

Suppose that $\mathcal{G}$ contains a perfect matching transversal $M$. One of the pairs $(A,C)$ or $(B,D)$ must contain $\frac{n}{2}$ edges of $M$ from the graphs $G_1, \dots, G_{n-1}$. This implies, however, that every edge of $E(G_n)$ must be incident to an edge $e$ from a graph $G_i$, $1 \leq i \leq n-1$, for which $e$ belongs to $M$. It follows that $M$ is not a perfect matching, giving a contradiction.

In the case that $n$ is odd, we may use a similar construction in which $|A| = |C| = \frac{n}{2} + \frac{1}{2}$ and $|B| = |D| = \frac{n}{2} - \frac{1}{2}$ to show that a minimum degree of $\frac{n}{2} - \frac{1}{2}$ is not sufficient to guarantee a perfect matching transversal. Thus the minimum degree condition of Theorem \ref{thmPM} is in a way best possible.

\section{Conclusion}
The condition that we give in Theorem \ref{mainThm} for a Hamiltonian transversal is a bipartite analogue to a result of Joos and Kim \cite{Joos}, which generalizes Dirac's theorem into the language of graph transversals. Similarly, the condition given in Theorem \ref{thmPancyclic} for bipancyclicity with respect to partial transversals is a bipartite analogue to a result of Cheng, Wang and Zhao \cite{Cheng}, which asymptotically generalizes a pancyclic version of Dirac's theorem into the language of transversals. It is well known, however, that Dirac's condition for both Hamiltonicity and pancyclicity can be extended in many different ways to more general conditions on vertex degree sums and degree sequences (see, for example, \cite{Ore}, \cite{Fan}, \cite{Bondy}). Similarly, for bipartite graphs, Theorems \ref{thmMoon} and \ref{thmSchmeichel} have extensions to more general conditions for vertex degree sums and degree sequences (see, for example, \cite{Moon}, \cite{Chvatal}, \cite{Hendry}). 

A question that remains open is whether these more general conditions for Hamiltonicity and (bi)pancyclicity can be translated into the language of graph transversals. For example, is there a theorem for graph transversals that directly generalizes Ore's theorem \cite{Ore} or Fan's condition \cite{Fan} for Hamiltonicity? While such theorems for graph transversals may exist, their proofs seem to be much more difficult than their traditional counterparts. 

\section{Acknowledgements}
I am grateful to Ladislav Stacho for making a suggestion to consider problems of pancyclicity in the setting of graph transversals and for carefully reading this paper. I am also grateful to Kevin Halasz for explaining key concepts and tools used in the study of graph transversals and for many helpful discussions. Additionally, I am grateful to Jaehoon Kim for pointing out the connection between Theorem \ref{thmPM} and the main result from \cite{AharoniRGraphs}, which ultimately led to a slight strengthening of the main three theorems of this paper. Finally, I am grateful to the referees who offered many helpful suggestions that improved the overall quality of this paper.

\raggedright
\bibliographystyle{abbrv}
\bibliography{MasterBib}

\end{document}